\documentclass{article}
\usepackage{amsfonts}
\usepackage{amsxtra}
\usepackage{amsthm}
\usepackage{amssymb}
\usepackage{verbatim}

\newtheorem{fed}{\textbf{Definition}}[section]

\newtheorem{lemma}[fed]{\textbf{Lemma}}

\newtheorem{ex}[fed]{\textbf{Example}}
\newtheorem{rem}[fed]{\textbf{Remark}}
\newtheorem{prop}[fed]{\textbf{Proposition}}

\newtheorem{cor}[fed]{\textbf{Corollary}}
\usepackage{amscd,amssymb,bbm,graphicx,epsfig,psfrag,epic,eepic,latexsym,amsmath}
\usepackage{amsmath}
\usepackage[arrow,matrix,curve]{xy}
\usepackage{mathrsfs}
\usepackage{graphicx}
\usepackage{hyperref}
\usepackage{color}

\newcommand{\N}{\mathbb{N}}

\newcommand{\Z}{\mathbb{Z}}
\newcommand{\R}{\mathbb{R}}
\newcommand{\C}{\mathbb{C}}

\newcommand{\wt}{\widetilde}
\newcommand{\JJ}{\mathcal{J}}
\newcommand{\eps}{\varepsilon}
\newcommand{\p}{\partial}

\usepackage[font=small,skip=0pt]{caption}

\begin{document}

\title{$J^+$-invariants for planar two-center Stark-Zeeman systems}
\author{Kai Cieliebak\thanks{Universit\"at Augsburg, kai.cieliebak@math.uni-augsburg.de, partially supported by DFG grant CI 45/8-1.}, Urs Frauenfelder\thanks{Universit\"at Augsburg, urs.frauenfelder@math.uni-augsburg.de, partially supported by DFG grant FR 2637/2-2}., Lei Zhao\thanks{Universit\"at Augsburg, lei.zhao@math.uni-augsburg.de, partially supported by DFG grant ZH 605/1-1.}}
\maketitle

\begin{abstract}
In this paper, we introduce the notion of planar two-center Stark-Zeeman systems and define four $J^{+}$-like invariants for their periodic orbits. The construction is based on a previous construction for planar one-center Stark-Zeeman system in~\cite{cieliebak-frauenfelder-koert} as well as Levi-Civita and Birkhoff regularizations. We analyze the relationship among these invariants and show that they are largely independent, based on a new construction called interior connected sum. 
\end{abstract}

\tableofcontents

\parindent=0pt
\parskip=4pt

\section{Introduction}

The notion of a (planar) Stark-Zeeman system was introduced in~\cite{cieliebak-frauenfelder-koert}. It describes the motion of
an electron in the plane attracted by a proton and subject to exterior electric and magnetic fields. Since Newton's law of gravitation takes the same form as  Coulomb's law, 
we can as well think of the electron as a light body gravitationally attracted by a proton as the heavy body. The Lorentz force from the magnetic field in this interpretation then corresponds to the Coriolis force. Many important systems from classical and celestial mechanics are Stark-Zeeman systems.

In a Stark-Zeeman system, the electron can collide with the proton, which causes singularities. Despite of this, it is classically known that such singularities due to two-body collisions can be regularized. In~\cite{cieliebak-frauenfelder-koert},  two invariants $\mathcal{J}_1$ and $\mathcal{J}_2$ were defined for families of regularized periodic orbits in Stark-Zeeman systems as immersed planar curves without direct self-tangency, based on Arnold's $J^+$-invariant \cite{arnold}, one for the unregularized system and another one for its Levi-Civita regularization.


In this paper we introduce the notion of a (planar) two-center Stark-Zeeman system. In this case the electron is attracted by two protons and the energy is high enough that the electron can collide with both of them, but not so high that the electron may escape from being close enough to the protons. An example of a two-center Stark-Zeeman system is the restricted three-body problem for energies between the first and second critical value. 

One of our motivations for defining $J^{+}$-type invariants of planar periodic orbits is to gain a better understanding about whether periodic orbits in given Stark-Zeeman systems
can be put in families of interpolating Stark-Zeeman systems. We shall introduce four $J^+$-like invariants for periodic orbits in a planar two-center Stark-Zeeman system. The generalization of the invariant $\mathcal{J}_1$ is straightforward. Since we have now two protons, we can consider the Levi-Civita regularization at either one of them.
This leads to two generalizations of the invariant $\mathcal{J}_2$ which we will refer to as $\mathcal{J}_E$ and
$\mathcal{J}_M$. The reason for this terminology is that in the interpretation of the restricted three-body problem 
one proton corresponds to the earth $E$ and the other proton corresponds to the moon $M$. For two-center Stark-Zeeman systems
there is a regularization due to Birkhoff which simultanuously regularizes the collisions with both primaries, i.e., with the
Earth and the Moon. The Birkhoff regularization gives rise to a fourth pair of invariants which we refer to as $(\mathcal{J}_{E,M}, n)$. We also analyze their relationships: depending on the parity of the winding numbers around $E$ and $M$ as well as their sums, sometimes one may express one of the invariants in terms of the others, while they are largely independent otherwise. The analysis is based on a construction called interior connected sum, which can be thought of as the inversion of the connected sum construction of a homotopically nontrivial immersed loop with an exterior homotopically trivial loop.


\section{Two-center Stark-Zeeman systems}\label{sec:2C}

Let $E, M \in \mathbb{R}^2 \cong \C$ be two distinct points which
we refer to as the Earth and Moon. Suppose that $\mu_E, \mu_M>0$. Let
$$V_E \colon \mathbb{R}^2 \setminus \{E\} \to \mathbb{R} \quad
q \mapsto -\frac{\mu_E}{|q-E|}, \qquad V_M \colon
\mathbb{R}^2 \setminus \{M\} \to \mathbb{R}, \quad
q \mapsto -\frac{\mu_M}{|q-M|}$$
be the gravitational potentials centered at the Earth and the Moon respectively.
The parameters $\mu_E$ and $\mu_M$ thus represent the masses of the Earth and the Moon respectively. 
Alternatively one may think of $V_E$ and $V_M$ as Coulomb potentials
under which the interpretations of the parameters $\mu_E$ and $\mu_M$ become charges. 

Assume that $U_0 \subset \mathbb{R}^2$ is an open set containing 
$E$ and $M$ and
$$V_1 \colon U_0 \to \mathbb{R}$$
is a smooth function.  Abbreviate
$$U:=U_0 \setminus \{E,M\}$$
and define
$$V:=V_E+V_M+V_1 \colon U \to \mathbb{R}.$$
The function $V_1$ can be interpreted as an additional potential which gives rise to additional position-dependent forces other than the gravitational forces of the Earth and the Moon. 

Velocity-dependent forces like the Lorentz force of a magnetic field or the Coriolis force can be modelled by a twist in the standard symplectic form of the cotangent bundle of $U$: 
For a function $\mathcal{B} \in C^\infty(U_0,\mathbb{R})$, let
$$\sigma_\mathcal{B}=\mathcal{B} \, dq_1 \wedge dq_2 \in \Omega^2(U_0)$$
and define the twisted symplectic form
$$\omega_\mathcal{B}=\sum_{i=1}^2 dp_i \wedge dq_i+\pi^* \sigma_\mathcal{B} \in 
\Omega^2(T^* U_0),$$
where $\pi \colon T^*U_0 \to U_0$ is the footpoint projection. 

We further choose a smooth Riemannian metric $g$ on $TU_0$. 
Let $g^*$ be its dual metric on the cotangent bundle $T^{*} U_{0}$ of $U_0$. We define the Hamiltonian
$$H=H_{V,g} \colon T^*U \to \mathbb{R}, \quad
(q,p) \mapsto \frac{1}{2}\|p\|^2_{g^*_q}+V(q).$$
The dynamics of the Stark-Zeeman system is given by the flow of the Hamiltonian vector field $X^B_{V,g}$ implicitly defined by
$$dH_{V,g}=\omega_\mathcal{B}(\cdot, X^\mathcal{B}_{V,g}).$$
As the Hamiltonian  is autonomous (i.e., independent of time), it is preserved under the flow of its Hamiltonian vector field (conservation of energy). We fix an energy value $c \in \mathbb{R}$ and consider a connected component
$$\Sigma_c\subset H^{-1}(c)$$
of the energy hypersurface on level $c$. The \emph{Hill's region} is defined as its image under the footpoint projection
$$\mathfrak{K}_c=\pi(\Sigma_c)\subset\{q \in U\mid V(q) \leq c\}.$$
We make the following two assumptions:
\begin{description}
 \item[C(i)] $c$ is a regular value of $H$ (or equivalently of $V$);
 \item[C(ii)] $\mathfrak{K}_c\cup\{E,M\}$ is bounded and simply connected.
\end{description}

\section{Examples of planar 2-center Stark-Zeeman systems}

{In this section we present a short list of classical planar 2-center Stark-Zeeman systems.}

\subsection{The planar circular restricted three-body problem}
A first system which fits into this category is the planar circular restricted three-body problem in a rotating frame so that $E$ and $M$ are fixed at the positions $(-\mu_M,0)$ and $(\mu_E,0)$, respectively. It is described by the Hamiltonian
$$H=\dfrac{|p|^{2}}{2} + V_{E} + V_{M} + V_{1}$$
with masses $\mu_{E},\mu_{M}>0$, which we can normalize by {setting} $\mu_{E}+\mu_{M}=1$. Here $V_{1}=\dfrac{|q|^{2}}{2}$ is the potential which generates the centrifugal force around the center of mass of $E$ and $M$, and the Coriolis force in the rotating frame is taken into account by the twisted symplectic form
$$\omega_{\mathcal{B}}=d (p_{1}-q_{2}) \wedge d q_{1} + d (p_{2} + q_{1}) \wedge d q_{2} = d p_{1} \wedge d q_{1} + d p_{2} \wedge d q_{2} + 2 d q_{1} \wedge d q_{2}.$$
There is a vast literature on this problem which we will not even try to list. Let us just mention that when the energy of the system is below the first critical value the Hill's region has three connected components: one around the Earth, one around the Moon, and another one ``around infinity''. When the energy $c$ lies between the first and the second critical values (counted from below), the two bounded connected components around the Earth and the Moon merge into one bounded component $\Sigma_c$ of the energy hypersurface satisfying assumptions $\mathbf{C(i)}$ and $\mathbf{C(ii)}$. In this case the corresponding Hill's region is actually homeomorphic to the connected sum of two discs, each with a point removed.
Above the second critical value, assumption $\mathbf{C(ii)}$ no longer holds.

\subsection{The charged planar circular restricted three-body problem}
The system is defined as in the planar circular restricted three-body problem, except that we no longer require $\mu_{E}, \mu_{M}$ to be positive. Instead they can be either positive or negative.  Such a system then models the motion of a charged particle in a magnetic field {and the electric field generated by the two charges. }
Note that when $\mu_{E}, \mu_{M}$ are not both positive at least one of the force fields is repulsive. Therefore, such a system on a fixed regular energy hypersurface may not satisfy assumption $\mathbf{C(ii)}$. 

\subsection{Euler's two-center problem in the plane}
Euler's two-center problem describes a particle moving in the gravitational field generated by two fixed bodies (the centers).  In the plane this corresponds to the case where $\mu_{E}, \mu_{M}>0$, $V_{1}\equiv 0$, and $\omega_{B}=\omega$ is the standard symplectic form. 
It was already known to Euler~\cite{euler} that this problem is separable in suitable coordinates and thus integrable. Regular energy hypersurfaces above the first critical value with negative energy satisfy assumptions $\mathbf{C(i), C(ii)}$, while regular energy hypersurfaces with positive energy satisfy assumption $\mathbf{C(i)}$ but not $\mathbf{C(ii)}$.


\subsection{Lagrange's modification of Euler's two-center problem}
The (planar) Lagrange problem is obtained from Euler's two-center problem by adding a quadratic potential $V_{1}=\dfrac{|q|^{2}}{2}$ at the midpoint of the two centers (which we may put at the origin). By the analysis of Lagrange \cite{Lagrange},  this system is also integrable.  

\subsection{Euler's problem and Lagrange's modification on a sphere or pseudosphere} 
Euler's two-center problem in the plane admits a generalization to the sphere and the pseudosphere, with the two-body potential replaced by $\mu \cot(\theta)$ and $\mu\, \hbox{coth}(\theta)$, respectively. {The system on the pseudosphere was defined and discussed in~\cite{Killing}, see also~\cite{vozmischeva}.} 
On the sphere the antipodal point of each center is again a center, with the strength constant $-\mu$. There are thus overall four centers on the sphere, two attractive and two repulsive. 

A new interpretation of the integrability of Euler's problem on the plane from the existence of Euler's problem on the sphere via central projection was {established} by A.~Albouy~\cite{albouy}. He actually realized both problems as quasi-bi-Hamiltonian systems, i.e., systems admitting two different Hamiltonian descriptions up to a time change. The projection of the spherical Hamiltonian then becomes a second conserved quantity of the planar system and vice versa. Moreover, in a gnomonic chart ({given by the central projection from the center of the sphere}) the spherical system takes the form of a Stark-Zeeman system with exactly the same potential as the planar system, just with a different kinetic energy. Lagrange's modification has also been discussed within this approach~\cite{albouy}. 
These systems in a gnomonic chart thus provide examples of two-center Stark-Zeeman systems with non-standard kinetic parts. Note that if instead we use a chart defined by stereographic projection, then in this chart the metric is conformal to the Euclidean metric and the singularities of these systems are asymptotically of Newtonian type, which allows us to treat these systems as examples of two-center Stark-Zeeman systems to which all the discussion below will apply.

\section{Partial and simultaneous regularizations of double collisions in planar 2-center Stark-Zeeman systems}

For a (planar) two-center Stark-Zeeman system, energy hypersurfaces which project to bounded Hill's regions are still noncompact due to the presence of collisions with the primaries. Nevertheless, we know that such collisions can be regularized, either individually or simultaneously. In this section we shall present {adaptations of} the Levi-Civita regularization for regularizing only one collision, and Birkhoff's simultaneous regularization of both collisions. There exist {also} other regularizations, but the Levi-Civita and Birkohoff regularizations are most suitable for our investigation of closed orbits in these systems via invariants of immersed planar loops. 

\subsection{Partial Levi-Civita regularizations}
We recall the Levi-Civita regularization of the planar Kepler problem. After normalization of the masses, the Hamiltonian of the system is given by 
$$H(q,p)=\dfrac{|p|^{2}}{2}-\dfrac{1}{|q|}$$
for $(q,p) \in \C\setminus \{0\} \times \C$.
To regularize the singularity at $q=0$, we fix an energy $c=-f<0$ and consider the Hamiltonian flow on 
$\Sigma_c=H^{-1}(c)$.
We change time on this energy hypersurface by rescaling the Hamiltonian to
$$\wt{H}(q,p):=|q| \bigl(H(q,p)-c\bigr)=\dfrac{|q||p|^{2}}{2} +f |q| -1.$$
We now consider the complex square mapping
$$L:\C\setminus\{0\} \to \C\setminus\{0\}, \qquad z \mapsto z^{2}.$$
Its cotangent lift is the symplectomorphism
$$T^*L: \C \setminus \{0\} \times \C \to  \C  \setminus \{0\} \times \C, \qquad (z, w) \mapsto (z^{2}, \dfrac{w}{2 \bar{z}}).$$
The regularized Hamiltonian $K$ is defined by pulling back $\wt{H}$ under $T^*L$,
$$K(z,w):=\wt H\circ T^*L (z,w)=\dfrac{|w|^{2}}{8} +f |z|^{2} -1. $$
The collision locus $\{q=0\}$ in the closure of $\Sigma_c$ is transformed to the set $\{z=0\}$ in the regular energy hypersurface $\{K=0\}$, which is no longer singular. These collisions are thus regularized. 

The Levi-Civita regularization extends to smoothly perturbed Kepler problems, in particular to all 1-center Stark-Zeeman systems. It applies also to 2-center Stark-Zeeman systems when we want to regularize only double collisions at either $E$ or $M$. We shall call these the {\em partial regularizations} with respect to $E$ and $M$ respectively. The other singularity remains non-regularized and, since the map $L$ is 2-to-1, the non-regularized singularity doubles to two singularities  in the partially regularized system. The two new singularities are still asymptotically of the type of a Newtonian type singularity: To see this, assume that the non-regularized singularity is located at $q=1$ and the potential is of the form $-\dfrac{1}{|q-1|}$. It contributes to the regularized system an additional term $-\dfrac{|z|^{2}}{|z^{2}-1|}=-\dfrac{|z|^{2}}{|z+1| |z-1|}$, so the two new singularities are located at $z=\pm 1$ and are of Newtonian type. We remark that this partial regularization procedure can thus be iterated, which is however not what we are going to investigate here.
In addition, we remark that the regularization procedure naturally extends to the case where the kinetic part of the Hamiltonian is given by a metric conformal to the standard Euclidean metric.

\subsection{Waldvogel's interpretation of Birkhoff's regularization}
We now present a regularization due to Birkhoff  {\cite{birkhoff}} of planar two-center Stark-Zeeman systems. 
By normalization, we put $E$ and $M$ at $-1$ and $1$, respectively. 

In \cite{waldvogel}, Waldvogel remarked that the complex square mapping $L(z)=z^{2}$
used in the Levi-Civita regularization extends to a conformal mapping from the Riemann sphere $\C \cup \{\infty\}$ to itself fixing $0$ and $\infty$ which, in Waldvogel's words \cite{waldvogel}, also ``regularizes'' 
a ``similar singularity'' at infinity. 
With this in mind, Waldvogel interpreted the Birkhoff regularization mapping  
\begin{equation}\label{eq:B}
   B:\C^*=\C\setminus\{0\}\to\C,\qquad B(z)=\frac{1}{2} (z + 1/z)
\end{equation}
%
%
as the conjugation $B = T^{-1} \circ L \circ T$ of the complex square mapping $L$ by the M\"obius transformation
$$
   T(z)=1-\dfrac{2}{1-z}=T^{-1}(z)
$$
sending $-1$ to $0$ and $+1$ to $\infty$. 
Thus $B$ extends to a branched double cover $\C\cup\{\infty\}\to\C\cup\{\infty\}$, sending $0$ and $\infty$ to $\infty$, with two branch points at $\pm 1$ of values $\pm 1$. See Figure \ref{fig:Birkhoff}.
\begin{figure}
\center
\includegraphics[width=100mm]{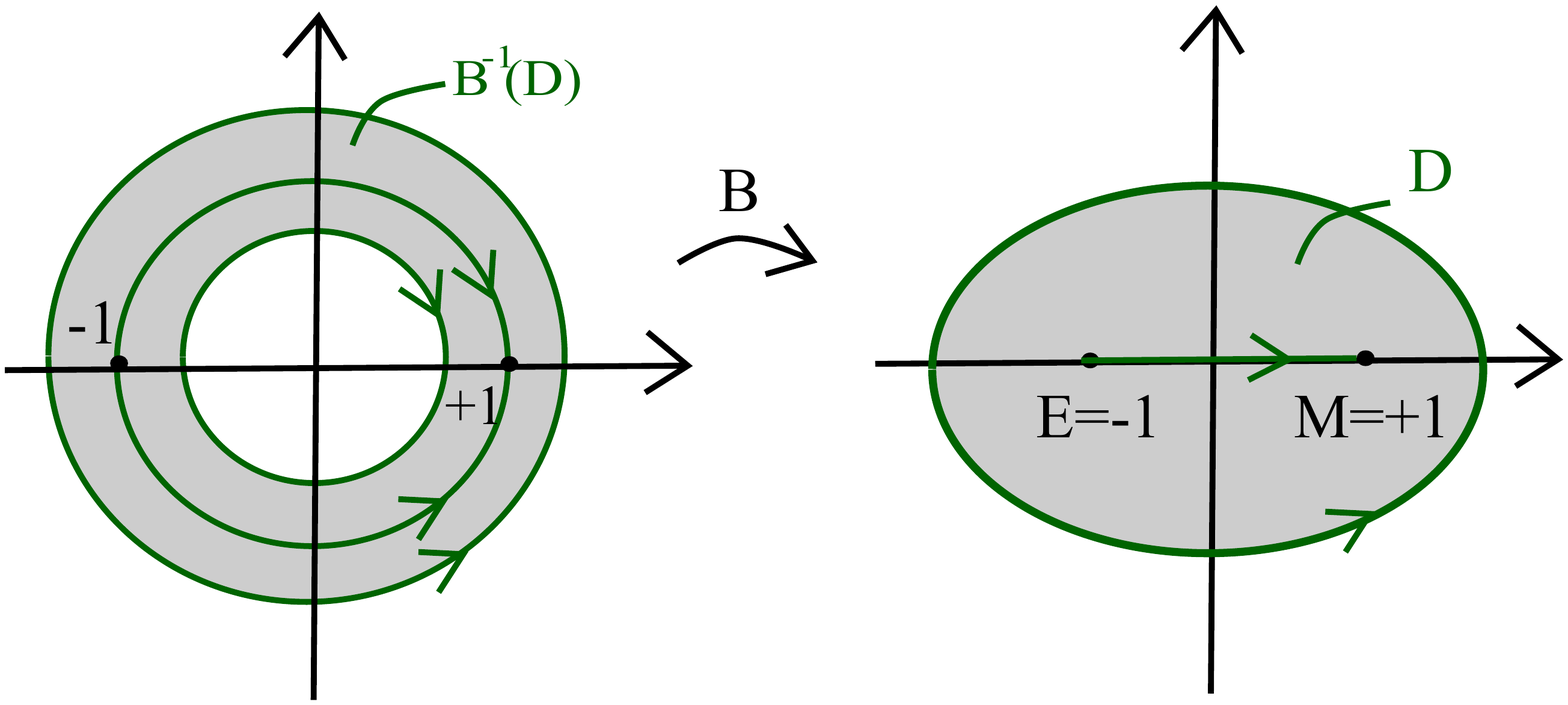}
\caption{Birkhoff regularization}
\label{fig:Birkhoff}
\end{figure}
The cotangent lift of $B$ is given by
\begin{equation}\label{eq:T*B}
   T^*B:T^*\C^*\to T^*\C,\qquad (z,w)\mapsto (q,p)=\Bigl(\frac{z^2+1}{2z},\dfrac{2 \bar{z}^{2}}{\bar{z}^{2}-1} w\Bigr).
\end{equation}
We will now explain the regularization of two-center Stark-Zeeman systems with this method, with Euler's two-center problem as a first example.

\subsection{Birkhoff simultaneous regularization of Euler's two-center problem}
In complex variables $(q,p) \in \C\setminus \{0, 1\} \times \C$, the Hamiltonian of the two-center problem is
$$H=\dfrac{|p|^{2}}{2}-\dfrac{\mu}{|q-1|}-\dfrac{1-\mu}{|q+1|}.$$
After fixing a negative energy {$c=-f$}
and rescaling time on this energy surface, we get that the slowed-down flow on this energy surface is governed by the following Hamiltonian restricted to the zero-energy level:
$$|q-1|\, |q+1| (H+f)=\dfrac{|q-1|\, |q+1||p|^{2}}{2}-\mu |q+1| -(1-\mu)|q-1| + f |q-1|\, |q+1|.$$
Substituting $(q,p)$ by $(z,w)$ via~\eqref{eq:T*B} and further dividing by $|z|^{2}$ results in the Hamiltonian
$$K(z,w)=\dfrac{|w|^{2}}{2} - \dfrac{ \mu |z+1|^{2}}{2 |z|^{3}} -  \dfrac{(1-\mu) |z-1|^{2}}{2 |z|^{3}}+f\dfrac{|z-1|^{2} |z+1|^{2}}{4 |z|^{4}}.$$
We observe that this system is no longer singular at the transformed collision sets $\{z=\pm 1\}$ in $\{K=0\}$. The Hamiltonian $K$ has a singularity at $z=0$, which however corresponds to energy $K=\infty$ and therefore does not lie on the energy hypersurface $\{K=0\}$. 
The regularized Hill's region, i.e.~the footpoint projection of the energy hypersurface $\{K=0\}$, is the subset in $\C$ described in polar coordinates $z=r e^{i \theta}$ by the inequality 
$$g_\theta(r):=2 r^{3} + 2 r -4 (1-2 \mu) r^{2} \cos \theta  -f (r^{2} -2 r \cos \theta +1)  (r^{2} +2 r \cos \theta +1)\geq0.$$

\begin{prop}\label{prop:Euler-Hill}
 For any $\mu \in (0, 1/2]$ there exists $f_{\mu}>0$ such that for all values $0<f<f_{\mu}$, the regularized Hill's region of the two-center problem at energy $-f$ is an annulus in $\C$ bounded by the boundaries of two star-shaped regions with respect to the origin.
\end{prop}

\begin{proof}
It suffices to show that the quartic equation $g_\theta(r)=0$
has exactly two positive real roots for any $\theta$. Let $\Delta_\theta$ be the discriminant of the quartic polynomial $g_\theta(r)$; an explicit formula of the discriminant in terms of the coefficients can be found at \url{https://en.wikipedia.org/wiki/Discriminant#Degree_4}. A calculation by Maple yields the factorization
$$\Delta_\theta=4096 f_{1}^{2} f_{2} f_{3}$$
where
$$f_{1}=1/4+f^2 \cos^{2} \theta+f (-1+2 \mu) \cos \theta,$$
$$f_{2}=f \cos^{2} \theta+(-1+2 \mu) \cos \theta-f-1,$$
$$f_{3}=f \cos^{2} \theta+(-1+2 \mu) \cos \theta-f+1.$$
We see that the discriminant is negative once $\mu\in(0,1/2]$ is fixed and $f$ is chosen small enough. This implies that there exist exactly two real roots for $g_\theta(r)$ and these real roots are distinct.  

To see that both of these real roots are positive, note that $\lim_{r \to +\infty} g_\theta(r)<0$ and $g_\theta(0)<0$. {On the other hand, a short calculation yields $g_\theta(1)>0$ for $f$ sufficiently small. }
Alternatively, we can use connectedness and non-contractibility of the regularized Hill's region asserted in Proposition~\ref{prop:4.2} below to conclude that there must exist some $r>0$ for which $g_\theta(r)>0$. Either way, we conclude that for any $\theta$ the polynomial $g_\theta(r)$ has exactly two positive roots. 
\end{proof}


\subsection{Birkhoff regularization of two-center Stark-Zeeman systems}

Consider now a general two-center Stark-Zeeman system as in Section~\ref{sec:2C} such that the {\em metric $g$ used in the kinetic energy is conformal to the standard metric}. Then replacing $p$ by $2\bar z^2w/(\bar z^2-1)$ yields $\|p\|_{g^*_q} = 2|z|^2\|w\|_{g^*_q}/|z^2-1|$ and the computation of the previous section goes through. Thus for a regular value $c$ satisfying conditions $C(i)$ and $C(ii)$ the level set $\Sigma_c\subset H^{-1}(c)$ pulls back under $T^*B$ to $\Sigma_c^B\subset K^{-1}(0)$ for the rescaled pullback Hamiltonian
$$
   K(z,w) = \dfrac{\|w\|_{g^*_q}^{2}}{2} - \dfrac{ \mu_M |z+1|^{2}}{2 |z|^{3}} -  \dfrac{\mu_E |z-1|^{2}}{2 |z|^{3}}+\dfrac{(V_1(q)-c)|z-1|^{2} |z+1|^{2}}{4 |z|^{4}},
$$ 
where $q$ needs to be replaced by $(z^2+1)/2z$. The singular point $z=0$ corresponds to $q=\infty$ which lies outside the closure $\bar{\mathfrak{K}}_c$ of the bounded Hill's region. So the hypersurface $\Sigma_c^B$ is regular and compact, and we call it the {\em Birkhoff regularization of $\Sigma_c$}. Note that the standard symplectic form twisted by a magnetic field $\sigma$ pulls back under $T^*B$ to the standard symplectic form twisted by the pullback magnetic field $B^*\sigma$. 

The footpoint projection of the Birkhoff regularized energy hypersurface $\Sigma_c^B$ is the preimage $B^{-1}(\bar{\mathfrak{K}}_c)$ under the map $B$ from~\eqref{eq:B}. Recall that we have normalized the positions of the Earth and Moon to $E=-1$ and $M=+1$; we denote the winding numbers around these points by $w_E$ and $w_M$, respectively. Then Proposition~\ref{prop:Euler-Hill} generalizes to

\begin{prop}\label{prop:4.2}
(a) The regularized Hill's region $B^{-1}(\bar{\mathfrak{K}}_c)\subset\C^*$ is an embedded annulus enclosing the origin.\\
(b) The preimage $B^{-1}(K)\subset\C^*$ of a closed curve $K\subset\C \setminus \{E, M\}$ is connected if $w_{E} (K) + w_{M} (K)$ is odd, and has two connected components if $w_{E} (K) + w_{M} (K)$ is even. 
\end{prop}

\begin{proof} 
Recall that map $B:\C^*\to\C$ from~\eqref{eq:B} is a branched double cover with two branch points at $\pm 1$ of values $\pm 1$. So each loop $K\subset\C\setminus\{-1,1\}$ lifts to a path in $\C^*$ which closes up iff $w_E(K)+w_M(K)$ is even. Part (b) immediately follows from this.
For part (a), note that $B$ maps the unit circle onto the interval $[-1,1]$, see Figure \ref{fig:Birkhoff}. 
Hence the preimage of an embedded circle $K\subset\C$ winding once around $-1$ and $+1$ consists of two disjoint embedded circles in $\C^*$ isotopic to the unit circle, and the preimage of any embedded disk $D\subset\C$ containing $-1$ and $+1$ (such as $D=\bar{\mathfrak{K}}_c$) is an embedded annulus in $\C^*$ enclosing the origin. 
\end{proof} 


Erdi~\cite{erdi} explains a way to deduce many other (known) regularizations of two-center Stark-Zeeman systems {(Le Maitre, Thiele-Burrau, Brouke, Wintner,\dots)} by composing the Birkhoff regularization with additional smooth transformations. The Birkhoff regularization is therefore a common basis to all these other regularizations. 

\subsection{Birkhoff versus Moser regularization}\label{Subsection: Top of Birkhoff reg}

We continue to use the notation from the previous subsection. 
Recall that the Birkhoff map $B(z)=(z+1/z)/2$ defines a double cover $B:\C^*\to\C$ branched at $E=-1$ and $M=+1$. It is invariant under the inversion $\phi(z)=1/z$ which interchanges the two sheets of the cover. Hence the cotangent lift $T^*B:T^*\C^*\to T^*\C$ of $B$ is invariant under the cotangent lift of $\phi$,
$$
   \Phi:=T^*\phi:T^*\C^*\to T^*\C^*,\qquad (z,w)\mapsto (z^{-1},-\bar z^2w). 
$$
By its construction as a compactification of $(T^*B)^{-1}(\Sigma_c)$, the Birkhoff regularized hypersurface $\Sigma_c^B$ is invariant under $\Phi$. (In fact, a direct computation shows $K\circ\Phi(z,w)=|z|^4K(z,w)$ for the Hamiltonian $K$ of the previous subsection.) Since the fixed points $(\pm 1,0)$ of $\Phi$ do not belong to $K^{-1}(0)$, the action of $\Phi$ on $\Sigma_c^B$ is free. So we obtain a quotient manifold $\Sigma_c^M$ and a 2-to-1 covering
\begin{equation}\label{eq:B-M}
   P:\Sigma_c^B\to \Sigma_c^M. 
\end{equation}
By construction, $\Sigma_c^M$ is a smooth compactification of the energy hypersurface $\Sigma_c$ and we call it the {\em simultaneous Moser regularization at $E$ and $M$}. Note that near each branch point $E,M$ the Birkhoff map looks like the Levi-Civita map around that point, so the 2-to-1 covering~\eqref{eq:B-M} is consistent with the 2-to-1 covering between the Levi-Civita and Moser regularizations of one-center Stark-Zeeman systems used in~\cite{cieliebak-frauenfelder-koert}. 

The following proposition describes the topology of the covering~\eqref{eq:B-M}. 

\begin{prop}\label{prop:B-M}
(a) There exist diffeomorphisms 
$$
   \Sigma_c^B\cong S^1\times S^2\quad\text{and}\quad \Sigma_c^M\cong\R P^3\#\R P^3
$$
such that the first diffeomorphism conjugates the involution $\Phi:\Sigma_c^B\to\Sigma_c^B$ to the map $S^1\times S^2\to S^1\times S^2$, $(\theta,u)\mapsto(-\theta,-u)$ (writing $S^1=\R/2\pi\Z$). \\
(b) The induced map between fundamental groups is given by
$$
   P_*:\pi_1(\Sigma_c^B)=\Z\to \pi_1(\Sigma_c^M)=\Z_2*\Z_2,\qquad n\mapsto(em)^n,
$$
where $e$ and $m$ are represented by lifts of small loops around $E$ and $M$, respectively. \\
(c) The free homotopy classes of loops in $\Sigma_c^M\cong\R P^3\#\R P^3$ correspond to the conjugacy classes $[e]$, $[m]$, and $[(em)^n]$ for $n\in\N_0$ in $\pi_1(\R P^3\#\R P^3)=\Z_2*\Z_2$. 
\end{prop}

\begin{proof}
(a) Recall that the closure of the Hill's region $\mathfrak{K}_c$ is a closed disk $D$ containing $E=-1$ and $M=1$, and its preimage $A:=B^{-1}(D)$ is a closed annulus enclosing the origin, see Figure \ref{fig:Birkhoff}. After deforming the Stark-Zeeman system (which does not affect the assertions of the proposition) we may assume that 
$$
   A = \{z\in\C\mid e^{-1}\leq|z|\leq e\} = \{z=e^{\rho+i\theta}\in\C\mid -1\leq\rho\leq 1\}.
$$
We use $(\rho,\theta)\in[-1,1]\times\R/2\pi\Z$ as coordinates on $A$, in which the inversion $\phi(z)=z^{-1}$ sends $(\rho, \theta)$ to $(-\rho, -\theta)$. The footpoint projection $\pi:\Sigma_c^B\to A$ defines a circle bundle over the interior of $A$ whose fibre circles collapse to points over the boundary $\partial A$ (the zero velocity curves). 
Thus for each fixed angle $\theta$ the preimage $\pi^{-1}([-1,1]\times\{\theta\})$ is a $2$-sphere, which gives the first diffeomorphism $\Sigma_c^B\cong S^1\times S^2$. 
Note that coordinates on $S^1\times S^2$ are given by $(\theta,u)$, where $\theta\in\R/2\pi\Z$ and $u=(\rho,w)\in[-1,1]\times\C$ with $\rho^2+|w|^2=1$.
Hence in these coordinates the map $\Phi(z,w)=(z^{-1},-\bar z^2w)$ takes (after rescaling $w$) the form 
$$
   \Phi:S^1\times S^2\to S^1\times S^2,\qquad \bigl(\theta,(\rho,w)\bigr)\mapsto\bigl(-\theta,(-\rho,-e^{-2i\theta}w)\bigr). 
$$
Conjugating $\Phi$ by the diffeomorphism 
$$
   \Gamma:S^1\times S^2\to S^1\times S^2,\qquad \bigl(\theta,(\rho,w)\bigr)\mapsto\bigl(\theta,(\rho,e^{-i\theta}w)\bigr)
$$
yields the desired map
$$
   \Gamma\Phi\Gamma^{-1}\bigl(\theta,(\rho,w)\bigr) 
   = \Gamma\Phi\bigl(\theta,(\rho,e^{i\theta}w)\bigr) 
   = \Gamma\bigl(-\theta,(-\rho,-e^{-i\theta}w)\bigr) 
   = \bigl(-\theta,(-\rho,-w)\bigr). 
$$
For the second diffeomorphism, we view $D$ as the boundary connected sum of two disks around $E$ and $M$. Then $\Sigma_c^M$ is the connected sum $\Sigma_E^M\#\Sigma_M^M$ of two Moser regularized energy hypersurfaces in one-center Stark-Zeeman systems, each of is diffeomorphic to $\R P^3$ as shown e.g.~in~\cite{cieliebak-frauenfelder-koert}. 
Alternatively, consider small closed disks $D_E,D_M\subset{\rm Int}\,D$ around $E,M$. Then $\pi^{-1}(D_E),\pi^{-1}(D_M)\subset\Sigma_c^M$ are solid tori and $\Sigma_c^M\setminus(\pi^{-1}(D_E)\amalg\pi^{-1}(D_M)$ is diffeomorphic to $S^3\setminus(T_E\amalg T_M)$ for unlinked and unknotted solid tori $T_E,T_M\subset S^3$. The local description of the Moser regularization near $E$ shows that to recover $\Sigma_c^M$, both $T_E$ and $T_M$ are glued in along their boundary by a diffeomorphism mapping the meridian to twice the meridian plus the longitude. Thus $\Sigma_c^M$ is the $2/1$-Dehn surgery of $S^3$ along two unlinked unknots (see e.g.~\cite{Gordon}), which equals $\R P^3\#\R P^3$. 

(b) By the description of the diffeomorphism $\Sigma_c^B\cong S^1\times S^2$ in (a), the outer boundary of $A$ represents a generator of $S^1$. Since it is mapped under $B$ onto $\partial D$, and $B$ lifts to $P$, this shows that $P_*$ maps a generator of $\pi_1(\Sigma_c^B)$ onto $em$. 

(c) Note that each element in $\Z_2*\Z_2$ is of the form $a_n=(em)^n$, $b_n=m(em)^n$ or $c_n=(em^n)e$ for some $n\in\N_0$. Since $mb_nm^{-1}=c_{n-1}$ and $ec_ne^{-1}=b_{n-1}$, all the elements $b_n,c_n$ are conjugated to either $e$ or $m$. 
\end{proof}

\begin{rem}
Proposition~\ref{prop:B-M} implies that the quotient of $S^1\times S^2$ under the fixed point free involution $\Phi(\theta,u)=(-\theta,-u)$ is diffeomorphic to $\R P^3\#\R P^3$. The geometry of the Birkhoff map leads to the following direct description of this diffeomorphism. Write
$$
   S^1=\R/2\pi\Z = I_0\cup I_2\cup I_3\cup I_4
$$
as the union of the four intervals
$$
   I_0= [-\frac{\pi}{4},\frac{\pi}{4}],\quad I_1=[\frac{\pi}{4},\frac{3\pi}{4}],\quad I_2= [\frac{3\pi}{4},\frac{5\pi}{4}],\quad I_3= [\frac{5\pi}{4},\frac{7\pi}{4}]
$$
glued at their endpoints. See Figure~\ref{fig:circle}.
\begin{figure}
\center
\includegraphics[width=100mm]{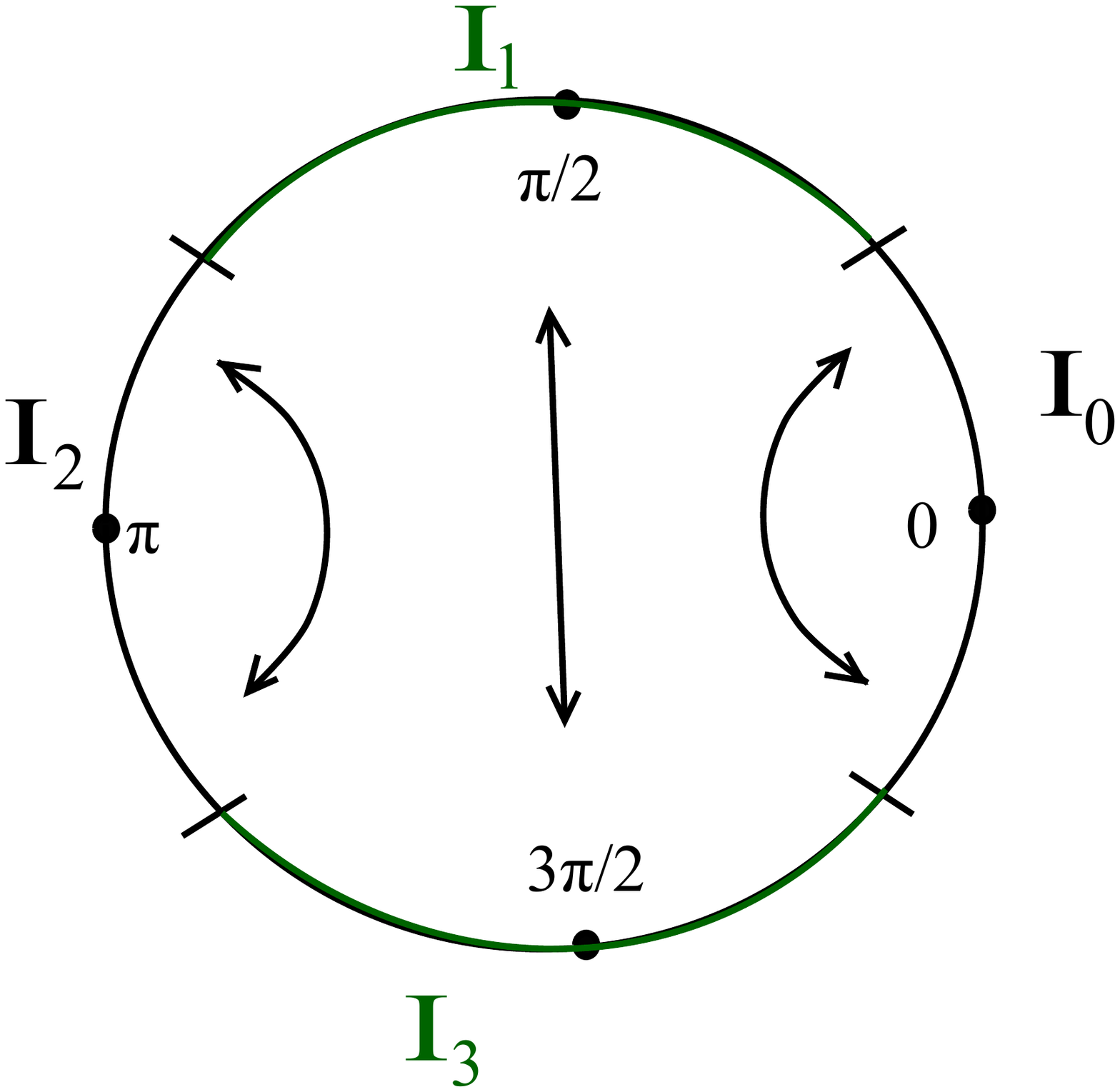}
\caption{The circle and the intervals}
\label{fig:circle}
\end{figure}
Note that the map $\theta\mapsto -\theta$ preserves $I_0$, $I_2$ and interchanges $I_1$ with $I_3$. Now we perform two $2$-surgeries on $S^1\times S^2$ along the spheres $\pi/2\times S^2$ and $3\pi/2\times S^2$, whose result can be explicitly written as (with the obvious gluings along the boundaries)
\begin{align*}
   N &:= \Bigl(S^1\times S^2\setminus(\mathring{I_1}\cup\mathring{I_3})\times S^2\Bigr)\cup(\p I_1\cup\p I_3)\times B^3 \cr
   &= (I_0\times S^2\cup\p I_0\times B^3)\amalg (I_2\times S^2\cup\p I_2\times B^3).
\end{align*}
Here $(I_0\times S^2\cup\p I_0\times B^3)\cong S^3$ and the involution $\Phi$ extends over $\p I_0\times B^3$ via $\Phi(\pm\pi/4,u)=(\mp\pi/4,-u)$. This gives the antipodal map on $S^3$, so its quotient is $\R P^3$ and the two balls $\p I_0\times B^3$ become one ball $\pi/4\times B^3$ in $\R P^3$. A similar discussion applies to the second component and we get
$$
   N/\Phi \cong \R P^3\amalg\R P^3
$$
with two distinguished balls $\pi/4\times B^3$ and $3\pi/4\times B^3$ in the two components. Now performing two $0$-surgeries on $N$ recovers
$$
   S^1\times S^2 = \Bigl(N\setminus(\p I_1\cup\p I_3)\times B^3\Bigr)\cup(I_1\cup I_3)\times S^2.
$$
Taking the quotient by $\Phi$, this yields
\begin{align*}
   S^1\times S^2/\Phi &= \Bigl(N/\Phi\setminus\p I_1\times B^3\Bigr) \cup I_1\times S^2 \cr
   &= \Bigl(\bigl(\R P^3\setminus\frac{\pi}{4}\times B^3\bigr)\amalg \bigl(\R P^3\setminus\frac{3\pi}{4}\times B^3\bigr)\Bigr) \cup I_1\times S^2 \cr
   &= \R P^3\#\R P^3.
\end{align*}
\end{rem}


\begin{rem}
The free product $\Z_2*\Z_2$ is isomorphic to the semidirect product $\Z_2\rtimes\Z$,
where $1\in\Z_2=\Z/2\Z$ acts on $\Z$ by $n\mapsto-n$. Indeed, we have the explicit isomorphism
$$\Z_2\rtimes\Z\stackrel{\cong}\longrightarrow \Z_2*\Z_2, \qquad (j,n) \mapsto (em)^n e^j.$$
\end{rem}

By Proposition~\ref{prop:B-M}(c), the free homotopy classes of loops in $\R P^3\#\R P^3$ (or equivalently, the connected components of its free loop space) are given by $[e]$, $[m]$, and $[(em)^n]$ for $n\in\N_0$. By Proposition~\ref{prop:B-M}(b), a loop in the class $[(em)^n]$ lifts under the covering map $P:S^1\times S^2 \to \mathbb{R}P^3\#\mathbb{R}P^3$ to two loops in $S^1\times S^2$, one representing
the conjugacy class $[n]$ and the other the class $[-n]$ in the fundamental group $\pi_1(S^1\times S^2)=\Z$. A loop in the class $[e]$ or $[m]$ does not lift to a loop in $S^1\times S^2$, but its double cover lifts to a contractible loop which is invariant under the involution $\Phi$. 

\subsection{A uniform view of partial and simultaneous regularizations}\label{Subsection: uniform view}

We have explained regularizations of either double collisions with one of the primaries or simultaneously for both. As Waldvogel's interpretation of the Birkhoff regularization suggests, we should consider these partial or simultaneously regularizations on the Riemann sphere which leads to a uniform view of them. We see that all of these regularization mappings are 2-to-1 complex {covering maps} branched at exactly two of the three points: $E, M, \infty$: The pair $(E, \infty)$ resp. $(M, \infty)$ gives rise to partial regularizations, while the pair $E, M$ gives rise to simultaneous regularizations. 

\section{$J^{+}$-invariants and Stark-Zeeman homotopies}

\subsection{Arnold's $J^{+}$-invariant for immersed loops in the plane}

In \cite{arnold}, Arnold defined three invariants $J^{+}, J^{-}, St$ for generic immersed loops in a plane. 
Here genericity means that there are only transverse double self-intersections.  Along a generic family of immersed loops three types of ``disasters'' may happen, direct and inverse self-tangencies and triple self-intersections, which give rise respectively to three quantities $J^{+}, J^{-}, St$. Of these quantities, $J^{+}$ is invariant under inverse self-tangiencies and triple self-intersections, while it increases by $2$ during a {\em positive} passage (i.e., such that two new double points are created) through a direct self-tangency. It is defined uniquely by these requirements and the normalizations on the {\em standard curves $K_j$} shown in Figure~\ref{fig:standard-curves}: it is normalized to $0$ on a figure-eight curve $K_0$, and to $2 -2|j|$ on the circle $K_{j}$ with $|j|-1$ interior loops and rotation number $j\in\Z$. 

%

\begin{figure}
\center
\includegraphics[width=100mm]{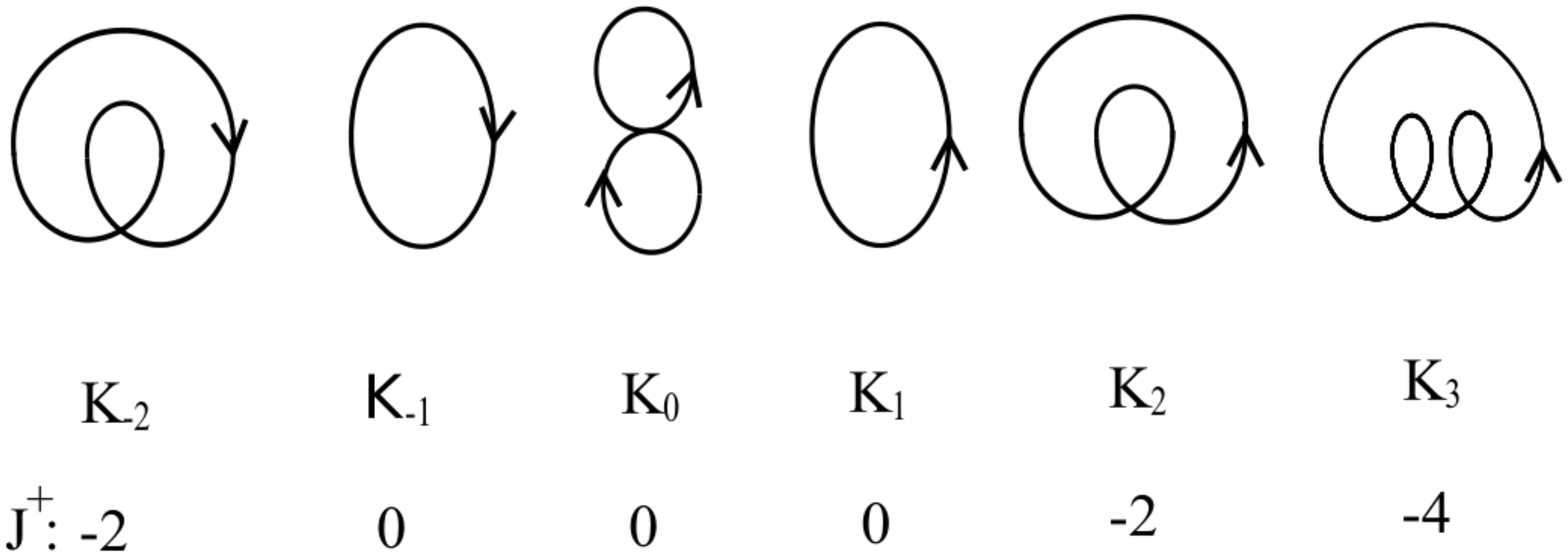}
\caption{The standard curves and their $J^+$-invariants}
\label{fig:standard-curves}
\end{figure}

Once we fix the energy in a Stark-Zeeman system, a direct self-tangency implies equality of the initial conditions and thus cannot happen for simple periodic orbits. The invariant $J^{+}$ is therefore relevant for periodic orbits of Stark-Zeeman systems. Assertion (a) of the following proposition is proved in~\cite{arnold} and assertions (b), (c) in~\cite{cieliebak-frauenfelder-koert}, where $w_{0}(K)$ denotes the winding number of a loop $K\subset\C\setminus\{0\}$ around the origin.

\begin{prop}\label{prop:J+}
(a) The invariant $J^{+}$ is independent of the orientation of the generic immersed loop $K\subset\C$, and additive under connected sum. \\
(b) Under addition of a loop in a component $C$ of $\C\setminus K$ to an arc $A\subset K$ the invariant changes by $-2w(K,C)$, where $w(K,C)$ is the winding number of $K$ around $C$ and $K$ is oriented by orienting $A$ as a boundary arc of $C$. \\
(c) For any pair of numbers $(n_1,n_2) \in 2\Z\times \Z$ there exists a generic immersed loop $K\subset\C\setminus\{0\}$ with $J^{+}(K)=n_1$ and $w(K)=n_{2}$. \hfill$\square$
\end{prop}
%

If we are given two distinct points $E, M \in \C$ and denote by $w_{E}(K), w_{M}(K)$ the corresponding winding numbers, then by taking the connected sum of two curves which wind around $E$ or $M$ with given total $J^{+}$ we obtain

\begin{cor} 
For any triple of numbers $(n_{1}, n_{2}, n_{3}) \in 2 \Z\times \Z\times\Z$ there exists a generic immersed loop $K\subset\C\setminus\{E,M\}$ with $J^{+}(K)=n_{1}$, $w_{E}(K)=n_{2}$ and $w_{M}(K)=n_{3}$. \hfill$\square$
\end{cor}

\subsection{Spherical $J^{+}$ for immersed loops on the sphere}

In \cite{arnold sphere}, Arnold defined a spherical analogue of the $J^{+}$-invariant for generic immersed loops on the sphere as follows. For a generic oriented immersed loop $K$ in the plane let $r(K)$ denote its {\em rotation number}, i.e., the degree of its normalized velocity vector $S^1\to S^1$, and define the 
{\em spherical $J^+$-invariant}
$$
   SJ^{+}(K) := J^{+}(K) + r(K)^{2}/2.
$$

\begin{prop}[Arnold~\cite{arnold sphere}]\label{prop: spherical j+}
$SJ^+$ induces a $J^+$-type invariant for generic immersed loops on the $2$-sphere. Moreover, it is invariant under diffeomorphisms of the sphere (in particular under M\"obius transformations).
\end{prop}

The first assertion means that if for a generic immersed loop $K$ on the sphere we remove a point from its complement and define $SJ^+(K)$ by the formula above for the resulting curve in the plane, then the definition does not depend on the choice of the point. Moreover, the resulting invariant for generic immersed loops on the sphere does not change under passage through triple self-intersections and inverse self-tangencies, and it increases by $2$ under positive passage through a direct self-tangency. 

\begin{proof}
For the first assertion, we need to prove that the quantity $SJ^+(K)$ for $K\subset\C$ does not change as an exterior arc $A$ of $K\subset\C$ is pulled over the point at infinity to an arc which encloses the rest of the curve. Let us denote the resulting curve by $K'$, see Figure \ref{fig:spherical-J+}. 
\begin{figure}
\center
\includegraphics[width=100mm]{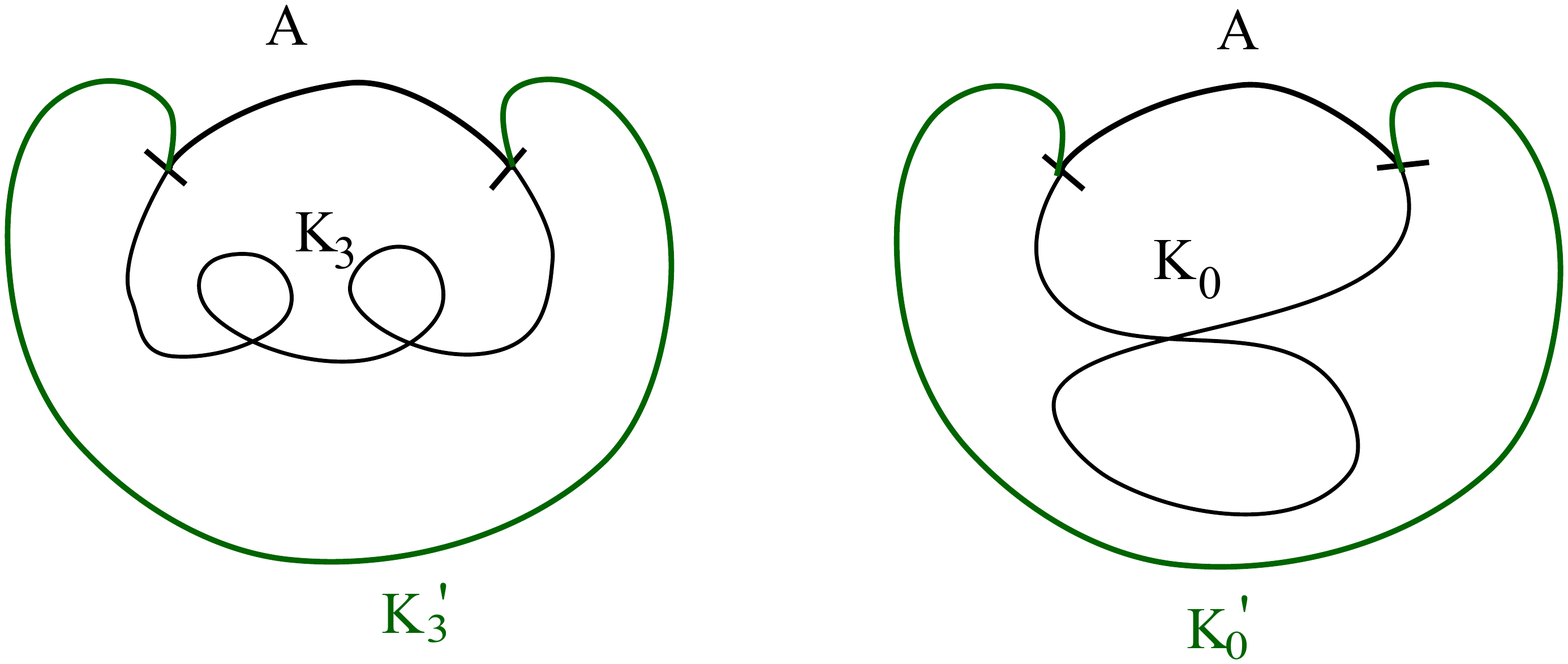}
\caption{Flipping an arc and the spherical $J^{+}$ invariant}
\label{fig:spherical-J+}
\end{figure}
By the proof of the Whitney--Graustein theorem~\cite{Whitney}, $K$ can be deformed to a standard curve $K_j$ by a regular homotopy keeping the arc $A$ fixed. Since $J^+(K)$, $J^+(K')$ change in the same way under this homotopy and $r(K)$, $r(K')$ remain unchanged, it therefore suffices to consider the case that $K=K_j$. Since $SJ^+(K)$ does not depend on the orientation of $K$, we may assume $r(K)=j\geq 0$. Suppose first that $j\geq 1$, so $K=K_j$ is a circle with $j-1$ interior loops. Then $K'$ is the standard curve $K_{-1}$ with $j-1$ {\em exterior} loops, and since by Proposition~\ref{prop:J+}(b) exterior loops do not affect $J^+$ we have $J^+(K')=0$. The rotation numbers are $r(K)=j$ and $r(K')=j-2$, so we get
$$
   SJ^+(K) = -2(j-1) + j^2/2 = (j-2)^2/2 = SJ^+(K').
$$ 
In the case $j=0$ we get $K'=K_{-2}$ and again $SJ^+(K') = -2+2^2/2 = 0 = SJ^+(K)$. This proves the first assertion. Invariance of $SJ^+$ under orientation preserving diffeomorphisms follows from homotopy invariance of $SJ^+$ and Smale's theorem~\cite{smale} that the group $\mathrm{Diff}^+(S^2)$ is homotopy equivalent to $SO(3)$ and therefore path connected. So it only remains to check invariance of $SJ^+$ under one orientation reversing diffeomorphism, e.g.~the reflection $R:\C\to\C$ at the $y$-axis. Since a regular homotopy from $K\subset\C$ to a standard curve $K_j$ gives a regular homotopy from $R(K)$ to $R(K_j)$ undergoing the same crossings through direct-self-tangencies, it suffices to consider the case $K=K_j$. But in this case invariance is obvious because we can choose $K_j$ so that $R(K_j)=K_j$, and the second assertion is proved. 
\end{proof}

We remark that the usual invariant $J^+$ for loops in the plane is invariant under planar diffeomorphisms, but for loops in $\C^*$ it is not invariant under the inversion $z\mapsto1/z$. 

\subsection{2-center Stark-Zeeman homotopies}

On a regular energy level set of a Stark-Zeeman system there is no equilibrium point, thus periodic orbits are nonconstant. Their footpoint projections fail to be an immersion only at collisions where velocity blows up, or at points on the boundary of the Hill's region (the ``zero-velocity curve'') where the velocity becomes zero. In~\cite{cieliebak-frauenfelder-koert} it is analyzed how these events can happen in a generic family of periodic orbits in a family of Stark-Zeeman systems, and it is shown that in either case the footpoint projections pass through a cusp with the creation/annihilation of a small loop. As these discussions are of local nature, the same holds for 2-center Stark-Zeeman systems, as well as for systems with singular potentials asymptotic to Newtonian ones such as partially regularized 2-center Stark-Zeeman systems.
Following~\cite{cieliebak-frauenfelder-koert}, we capture all these events in the following definition, where $E,M$ are two distinct points in $\C$. Here a closed curve is called {\em simple} if it is not multiply covered.

\begin{fed} 
A {\em 2-center Stark-Zeeman homotopy} is a smooth 1-parameter family $K^{s},\, s \in [0, 1]$ of simple closed curves in $\C$ which are generic immersions in $\C\setminus\{E,M\}$, except for finitely many $s\in [0,1]$ where the following events can occur (see Figures 5--8 in~\cite{cieliebak-frauenfelder-koert}): 
\begin{itemize}
\item $(I_E)$ birth or death of interior loops through cusps at $E$;
\item $(I_M)$ birth or death of interior loops through cusps at $M$;
\item $(I_{\infty})$ birth or death of exterior loops through cusps; 
\item $(II^-)$ crossings through inverse self-tangencies;
\item $(III)$ crossings through triple-self-intersections.
\end{itemize}
\end{fed}

The following proposition carries over directly from the corresponding result in~\cite{cieliebak-frauenfelder-koert} to the 2-center case. 

\begin{prop} 
A 1-parameter family $(K^{s})_{s \in [0, 1]}$ of simple closed curves in $\C \setminus \{E, M\}$ is a 2-center Stark-Zeeman homotopy if and only if there exists a smooth family of diffeomorphisms $F^{s}: \C \setminus \{E, M\} \to \C \setminus \{E, M\}$ such that, after suitable reparametrization, the curves $F^{s}(K^{s})$ are the footpoint projections of simple periodic orbits (possibly with collisions) in a generic family of 2-center Stark-Zeeman systems.
\end{prop}

The following lemma describes the topology of loops in $\C\setminus\{E,M\}$. Note that the group in (a) equals the fundamental group of the Moser regularized energy hypersurface $\Sigma_c^M\cong\R P^3\#\R P^3$ described in Proposition~\ref{prop:B-M}, the correspondence being given by the footpoint projection. 

\begin{lemma}\label{lem:top-loops}
(a) The fundamental group of $\C\setminus\{E,M\}$ modulo the moves $(I_E)$ and $(I_M)$ equals $\Z_2*\Z_2=\langle e,m\mid e^2=m^2=1\rangle$, where $e$ and $m$ correspond to loops around $E$ and $M$, respectively. \\
(b) The free homotopy classes of loops in $\C\setminus\{E,M\}$ modulo the moves $(I_E)$ and $(I_M)$ are the conjugacy classes $[e]$, $[m]$, and $[(em)^n]$ for $n\in\N_0$. \\
(c) The regular homotopy classes of immersed loops in $\C\setminus\{E,M\}$ modulo the moves $(I_E)$ and $(I_M)$ are classified by their free homotopy class as in (b) together with their rotation number.
\end{lemma}

\begin{proof}
Part (a) holds because the fundamental group of $\C\setminus\{E,M\}$ equals $\Z*\Z=\langle e,m\mid -\rangle$ and the moves $(I_E)$ and $(I_M)$ convert $e$ to $e^{-1}$ resp.~$m$ to $m^{-1}$. Part (b) follows from Proposition~\ref{prop:B-M}(c), and part (c) follows from the proof of the Whitney--Graustein theorem~\cite{Whitney}.   
\end{proof}

\section{$J^{+}$-like invariants for two-center Stark-Zeeman systems}

In this section we define four $J^{+}$-like invariants for two-center Stark-Zeeman systems and investigate the relations among these.
Throughout this section we assume that the metric entering the Stark-Zeeman Hamiltonian is conformal to the standard metric, so that the partial Levi-Civita regularizations at $E$ and $M$ as well as the Birkhoff regularization are defined. 

\subsection{$\mathcal{J}_{0}$ with no regularization}

First we will define a $J^{+}$-like invariant for periodic orbits of 2-center Stark-Zeeman systems without invoking any regularizations. 
Following~\cite{cieliebak-frauenfelder-koert}, the idea is to balance out the possible change of $J^{+}$ at ``disasters'' that a Stark-Zeeman homotopy may encounter by winding numbers. As we have two possible double collisions, we have to use both winding numbers around the Earth and Moon:

\begin{fed} We define
$$\mathcal{J}_{0} (K):=J^{+} (K) + w_{E} (K)^{2}/2 + w_{M} (K)^{2}/2$$
where $w_{E}$ and $w_{M}$ are respectively the winding numbers of the curve around $E$ and $M$. 
\end{fed}

\begin{prop}\label{prop: J_{0} inv}
The quantity $\mathcal{J}_{0}$ is invariant under Stark-Zeeman homotopies. 
\end{prop}

\begin{proof} Under the moves $(II^-)$ and $(III)$ all of the involved quantities $J^{+}, w_{E}, w_{M}$ are invariant, hence also $\mathcal{J}_{0}$. The same holds for the move $(I_{\infty})$ because $J^{+}$ as well as the winding numbers $w_{E}, w_{M}$ are invariant under connected sum with an exterior loop. For $(I_E)$, we know from~\cite[Proposition 4]{cieliebak-frauenfelder-koert} that at a birth or death of loops though cusps at $E$ the quantity $J^{+} + w_{E}^{2}/2$ is invariant, while $w_{M}^{2}/2$ is clearly invariant, therefore $\mathcal{J}_{0}$ is invariant. The same argument works for $(I_M)$. 
\end{proof}

\subsection{$\mathcal{J}_{E}, \mathcal{J}_{M}$ via partial regularizations}

We may regularize the double collisions with the primary $E$ (resp.~$M$) by Levi-Civita regularization.  In this partially regularized system, the other primary $M$ (resp.~$E$) is pulled back to two singularities that we denote by $M_{1}, M_{2}$ (resp.~$E_{1}, E_{2}$).  We denote by $\tilde{K}_{E}$ (resp.~$\tilde{K}_{M}$) a connected component of the preimage of a curve $K$ in the partially regularized system with respect to $E$ (resp.~$M$). 

\begin{fed}\label{def: J_{E}, J_{M}}We set
$$\mathcal{J}_{E} (K):=J^{+} (\tilde{K}_{E}) + w_{M_{1}} (\tilde{K}_{E})^{2}/2 + w_{M_{2}} (\tilde{K}_{E})^{2}/2,$$
$$\mathcal{J}_{M} (K):=J^{+} (\tilde{K}_{M}) + w_{E_{1}} (\tilde{K}_{M})^{2}/2 + w_{E_{2}} (\tilde{K}_{M})^{2}/2.$$
\end{fed}

\begin{prop}
The quantities $\mathcal{J}_{E}(K)$, $\mathcal{J}_{M}(K)$ do not depend on the choice of the connected components $\tilde{K}_{E}$, $\tilde{K}_{M}$ and are invariant under Stark-Zeeman homotopies. 
\end{prop}

\begin{proof}
We will do the proof for $\mathcal{J}_{E}$, which implies the one for $\mathcal{J}_{M}$ by switching the roles of $E$ and $M$. As in the proof of Proposition~\ref{prop: J_{0} inv}, $\mathcal{J}_{E}(K)$ is invariant under $(II^-)$, $(III)$ and $(I_{\infty})$. Invariance under $(I_E)$ holds because $\tilde K_E$ remains smooth under this move. For $(I_M)$, note that each passage of $K$ through a cusp at $M$ corresponds to a passage of $\tilde K_E$ through cusps at both $M_1$ and $M_2$ (if $w_E(M)$ is odd), or through a cusp at one of $M_1$, $M_2$ (if $w_E(K)$ is even). In either case, the change in $J^{+}(\tilde{K}_{E})$ is offset by the change in $w_{M_{1}}(\tilde{K}_{E})^{2}/2 + w_{M_{2}}(\tilde{K}_{E})^{2}/2$. This proves invariance of $\mathcal{J}_{E}$ under Stark-Zeeman homotopies. 
\end{proof}

The following lemma provides alternative expressions for $\JJ_E$ and $\JJ_M$. 

\begin{lemma}\label{lem:JE}
If $w_E(K)$ is odd, then
\begin{align*}
  \mathcal{J}_{E}(K) = J^{+}(\tilde{K}_{E}) + w_{M}(K)^{2}. 
  \end{align*}
If $w_E(K)$ is even and $K=K_1\#K_2$ is a connected sum of immersions $K_1$ and $K_2$ located near $E$ and $M$, respectively, then
\begin{align*}
  \mathcal{J}_{E}(K) = J^{+}(\tilde{K}_{E}) + w_{M}(K)^{2}/2. 
  \end{align*}
Analogous formulas hold for $\JJ_M$. 
\end{lemma}

\begin{proof}
Again, it suffices to consider $\JJ_E$. If $w_E(K)$ is odd, then the preimage $L_E^{-1}(K)$ of $K$ under the complex square map $L_E$ around $E$ is connected and $\tilde K_E=L_E^{-1}(K)$. We normalize the positions of the primaries to $E=0$, $M=1$ so that $L_E(z)=z^2$. Then the preimage under $L_E$ of the ray $[1,\infty)$ emanating from $M=1$ is the union of the rays $[1,\infty)$ emanating from $M_1=1$ and $(-\infty,-1]$ emanating from $M_2=-1$. Since each crossing of $K$ through the ray $[1,\infty)$ corresponds to crossings of $\tilde K_E$ though the rays $[1,\infty)$ and $(-\infty,-1]$ with the same sign, and the winding numbers are given by the signed counts of such crossings, it follows that $w_M(K)=w_{M_1}(\tilde K_E)=w_{M_2}(\tilde K_E)$. The formula $\mathcal{J}_{E}(K)=J^{+} (\tilde{K}_{E}) + w_{M} (K)^{2}$ is an immediate consequence of this. 

Now suppose that $w_E(K)$ is even and $K=K_1\#K_2$ is a connected sum of immersions $K_1$ and $K_2$ located near $E$ and $M$, respectively. Then $\wt K_E=\wt K_1\#\wt K_2$ for components $\wt K_i$ of $L_E^{-1}(K_i)$, $i=1,2$. Since $\wt K_1$ is located near $E$ and $\wt K_2$ near one preimage of $M$, say $M_1$, we have $w_{M_1}(\wt K_E)=w_M(K)$ and $w_{M_2}(\wt K_E)=0$, hence $\mathcal{J}_{E}(K) = J^{+}(\tilde{K}_{E}) + w_{M}(K)^{2}/2$. 
\end{proof}

\begin{ex}\label{ex:non-conn-sum}
Let {$K\subset\C\setminus\{E,M\}$} be an immersed loop winding twice counterclockwise around $E$ and $M$ with one self-intersection, see Figure \ref{fig:not-conn-sum}. 
\begin{figure}
\center
\includegraphics[width=100mm]{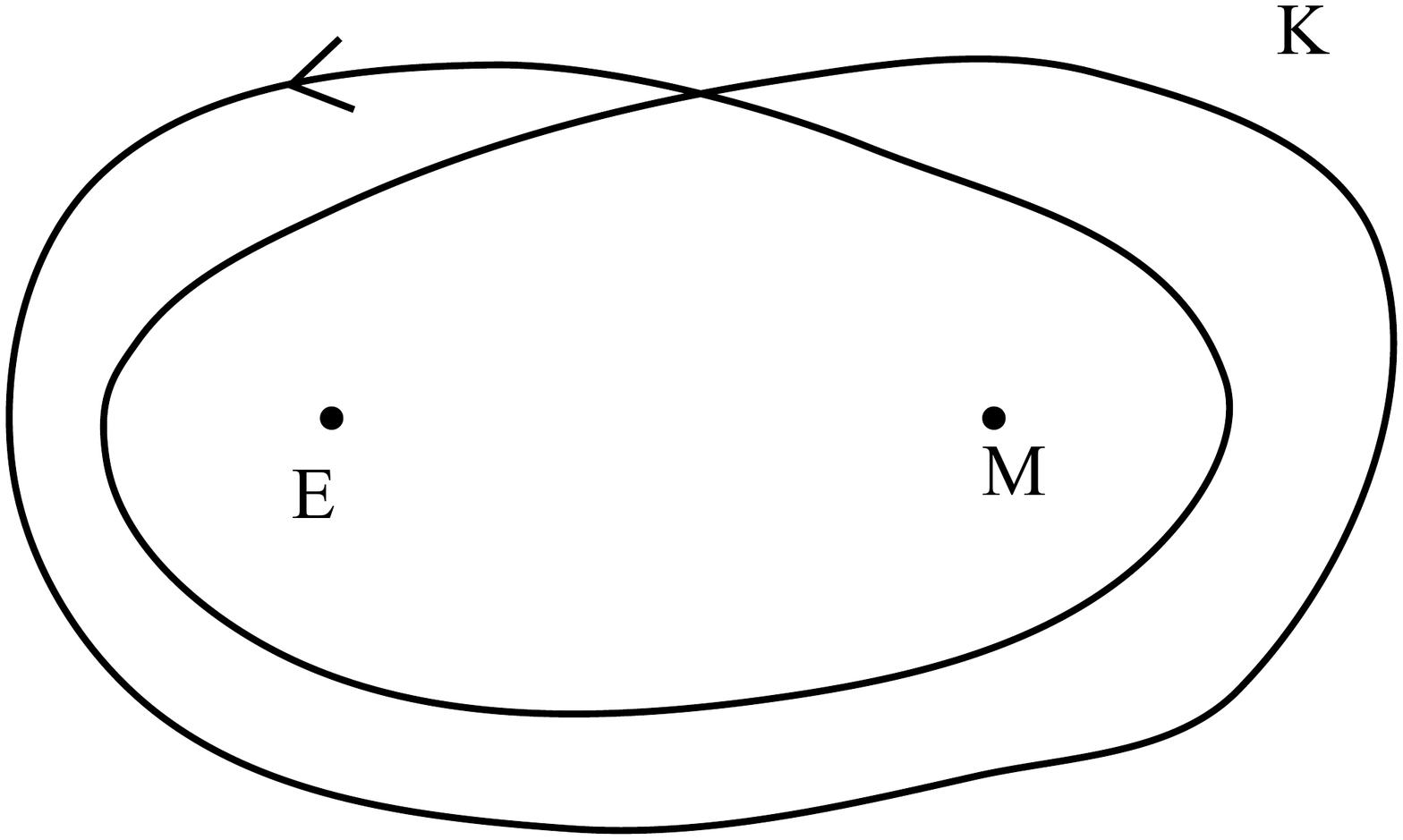}
\caption{A loop which is not a connected sum of loops around $E$ and $M$}
\label{fig:not-conn-sum}
\end{figure}
Then $\tilde K_E$ is an embedded loop winding once counterclockwise around $E,M_1,M_2$, so we have $w_E(K)=w_M(K)=2$ and $w_E(\tilde K_E)=w_{M_1}(\tilde K_E)=w_{M_2}(\tilde K_E)=1$. Hence $\mathcal{J}_{E}(K)=0+1/2+1/2=1$. Since the expression $J^{+}(\tilde{K}_{E}) + w_{M}(K)^{2}/2$ can never be an odd integer, this shows that the second assertion in Lemma~\ref{lem:JE} does not hold without the connected sum hypothesis. By invariance of $\JJ_E$, it also shows that this $K$ is not Stark-Zeeman homotopic to a connected sum of two immersed loops located near $E$ and $M$.   
\end{ex}

\subsection{{$(\mathcal{J}_{E, M}, n)$} via simultaneous regularization}

Consider now the Birkhoff regularization map $B:\C^*\to\C$, where we again choose $E=-1$ and $M=+1$. For a loop $K\subset\C\setminus\{E,M\}$ we denote by $\tilde K\subset\C^*$ one component of its preimage under $B$. 
%
%
%
Recall that the regularized Hill's region $B^{-1}(\mathfrak{K}_c)$ is an annulus winding around the origin and containing no more singularities. However, the invariant $J^+(\tilde K)$ may change under a Stark-Zeeman homotopy due to the addition of {\em interior} loops which are in the preimage of exterior loops added to the original curve $K$ under a $(I_\infty)$ move. Moreover, in the case that $B^{-1}(K)$ is disconnected its two preimages may have different $J^+$-invariants. Nevertheless, we can still extract an invariant from $J^+(\tilde K)$. 
\begin{fed}\label{defi J_{e, m}} 
For a generic immersed loop $K\subset\C\setminus\{E,M\}$, we choose a component $\tilde K\subset\C^*$ of its preimage under $B$ and set
$$
   n(K):=|w_{0}(\tilde{K})|\in\N_0.
$$ 
Moreover, we define 
$$
   J_{E, M}(K) := \begin{cases}
      \mathcal{J}^{+}(\tilde{K}) & \text{ if }n(K)=0, \cr
      \mathcal{J}^{+}(\tilde{K}) \mod 2 n(K) & \text{ if }n(K)>0.
   \end{cases}
$$
\end{fed}

{To show that these are well-defined, we shall need the following lemma:}

\begin{lemma}\label{lem:rot-win} 
If $B^{-1}(K)$ has two connected components $\tilde{K}_1,\tilde{K}_2$, then
$$
   r(\tilde{K}_{2})- r(\tilde{K}_{1}) = w_{0}(\tilde{K}_{2}) - w_{0}(\tilde{K}_{1}) = -2w_0(\tilde{K}_1).
$$
\end{lemma}

\begin{proof} 
%
Recall that $\tilde{K}_2=\phi(\tilde{K}_1)$ for $\phi(z)=1/z$. Thus a parametrization $z_1(t)$ of $\tilde{K}_1$ gives rise to a parametrization $z_2(t)=1/z_1(t)$ of $\tilde{K}_2$. This shows that $w_0(\wt K_1)=-w_0(\wt K_2)$. Moreover, the equation $\dot z_2(t) = -\dot z_1(t)/z_1(t)^2$ yields the relation $r(\tilde{K}_2)=r(\tilde{K}_1)-2w_0(\tilde{K}_1)$.  
\end{proof}

\begin{prop}\label{prop:JEM}
The quantities $n(K)$ and $\mathcal{J}_{E,M}(K)$ do not depend on the choice of $\tilde K$ and are invariant under two-center Stark-Zeeman homotopies. 
\end{prop} 

\begin{proof} 
Suppose that $B^{-1}(K)$ has two components $\tilde{K}_1,\tilde{K}_2$ (the proof in the case that $B^{-1}(K)$ is connected is similar but simpler and will be omitted).
Then by Lemma~\ref{lem:rot-win} we have $w_0(\tilde{K}_1)=-w_0(\tilde{K}_2)$, so $n(K)=|w_0(\tilde{K}_1)|=|w_0(\tilde{K}_2)|$ does not depend on the choice of a component. Moreover, $n(K)$ does not change under a Stark-Zeeman homotopy because $\tilde{K}_1,\tilde{K}_2$ never cross the origin. 

Since by Proposition~\ref{prop: spherical j+} the spherical $J^+$-invariant is preserved under M\"obius transformations, it is the same for $\tilde{K}_1$ and $\tilde{K}_2$, i.e.
$$
   J^{+}(\tilde{K}_{1}) + r(\tilde{K}_{1})^{2}/2 = J^{+}(\tilde{K}_{2}) + r(\tilde{K}_{2})^{2}/2.
$$
We rearrange this equation and invoke Lemma~\ref{lem:rot-win} twice to get
\begin{align}\label{eq: difference J+}
   J^{+}(\tilde{K}_{2})-J^{+}(\tilde{K}_{1}) &= \frac{r(\tilde{K}_{1})^{2}-r(\tilde{K}_{2})^{2}}{2} = \frac{\bigl(r(\tilde{K}_{1})-r(\tilde{K}_{2})\bigr)\bigl(r(\tilde{K}_{1})+r(\tilde{K}_{2})\bigr)}{2} \cr
   &= w_0(\tilde{K}_1)\,\bigl(r(\tilde{K}_{1}) + r(\tilde{K}_{2})\bigr) \cr
   &= 2w_0(\tilde{K}_1)\,\bigl(r(\tilde{K}_{1}) - w_0(\tilde{K}_{1})\bigr).
\end{align}
As the right hand side is an integer multiple of $2n(K)$, this shows that $J_{E,M}(K)$ does not depend on the choice of the component $\tilde K$. Moreover, it is clearly invariant under the moves $(I_E)$, $(I_M)$, $(II^-)$ and $(III)$ for $K$. A move $(I_\infty)$ for $K$ results in addition/removal to/from $\tilde K$ of an exterior loop, an interior loop in the component of $\C\setminus\wt K$ containing the origin, or both (if $B^{-1}(K)$ is connected). As an exterior loop does not change $J^+(\wt K)$ and an interior loop changes it by $-2w_0(\wt K)$, this proves invariance of $J_{E,M}(K)$ under Stark-Zeeman homotopies. 
\end{proof}

The following lemma shows that the parity of $n(K)$ is determined by that of $w_E(K)$ and $w_M(K)$. 

\begin{lemma}\label{lem:n}
If $w_E(K)+w_M(K)$ is odd, then $n(K)=0$. If $w_E(K)+w_M(K)$ is even, then $n(K)\equiv w_E(K)\equiv w_M(K)\mod 2$. 
\end{lemma}

\begin{proof}
Recall that we have normalized $E=-1$, $M=1$ and the Birkhoff map is given by $B(z)=(z+z^{-1})/2$. So $B$ maps the arcs $(1,\infty)$ and $(0,1)$ bijectively onto $(1,\infty)$, preserving the orientation for $(1,\infty)$ and reversing it for $(0,1)$ (where we always orient an arc $(a,b)$ from $a$ to $b$). We perturb $K\subset\C\setminus\{-1,1\}$ to make it transverse to the arc $(1,\infty)$. Then each intersection point $p$ of $K$ with $(1,\infty)$ corresponds to a pair $(p_+,p_-)$ consisting of an intersection point $p_+$ of $B^{-1}(K)$ with $(1,\infty)$ of the same sign, and an intersection point $p_-$ of $B^{-1}(K)$ with $(0,1)$ of opposite sign. Since the winding number of $B^{-1}(K)$ around the origin equals the signed count of its intersection points with $(0,\infty)$, this shows that $w_0\bigl(B^{-1}(K)\bigr)=0$ (and therefore $n(K)=0$) if $B^{-1}(K)$ is connected, i.e., if $w_E(K)+w_M(K)$ is odd. 

If $w_E(K)+w_M(K)$ is even, then $B^{-1}(K)$ consists of two components $\tilde{K}_1,\tilde{K}_2$. By the preceding discussion, each intersection point of $K$ with $(1,\infty)$ corresponds to an intersection point of $\tilde{K}_1$ with $(0,\infty)$ (possibly of different sign). So the winding numbers $w_M(K)$ of $K$ around $M=1$ and $w_0(\tilde{K}_1)$ of $\tilde{K}_1$ around $0$ have the same parity. 
\end{proof}

 \begin{rem}
The invariant $n(K)$ is uniquely determined by the free homotopy class of the (co-)tangent lift of $K$ to the Moser regularized energy hypersurface $\Sigma_c^M=\R P^{3} \# \R P^{3}$: As explained at the end of Subsection~\ref{Subsection: Top of Birkhoff reg}, a loop in the class $[(em)^n]$, $n\in\N_0$ lifts to two loops in the free homotopy classes $[\pm n]$ in the Birkhoff regularized hypersurface $\Sigma^B_c=S^1\times S^2$ and thus has $n(K)=n$, while a loop in the class $[e]$ or $[m]$ has its double cover lifting to a contractible loop in $S^1\times S^2$ and thus has $n(K)=0$. 
\end{rem}

\begin{ex}
Consider the two curves in Figure~\ref{Fig: Ke Kem}. Both curves $K_{E}$ and $K_{EM}$ have $J^{+}=2$ and winding numbers $w_{E}=w_{M}=0$. However, they are not Stark-Zeeman homotopic. To see this, note first that both curves are contractible in $\C\setminus\{E,M\}$, so the components of their preimages under the Birkhoff regularization map $B$ have winding number $0$ around the point $0$. Since the embedded arcs in $K_E$ connecting a self-intersection point have winding number $\pm 1$ around $E$ and $0$ around $M$, the self-intersection points disappear in $B^{-1}(K_E)$, hence $B^{-1}(K)$ is a union of two embedded loops and $\mathcal{J}_{E, M}(K_{E})=0$. By contrast, the embedded arcs in $K_{EM}$ connecting a self-intersection point have winding number $\pm 1$ around both $E$ and $M$, so the self-intersection points persist in $B^{-1}(K_{EM})$, hence each component of $B^{-1}(K_{EM})$ is diffeomorphic to $K_{EM}$ and $\mathcal{J}_{E, M}(K_{EM})=2$. 
\end{ex}

\begin{figure}
\center
\includegraphics[width=80mm]{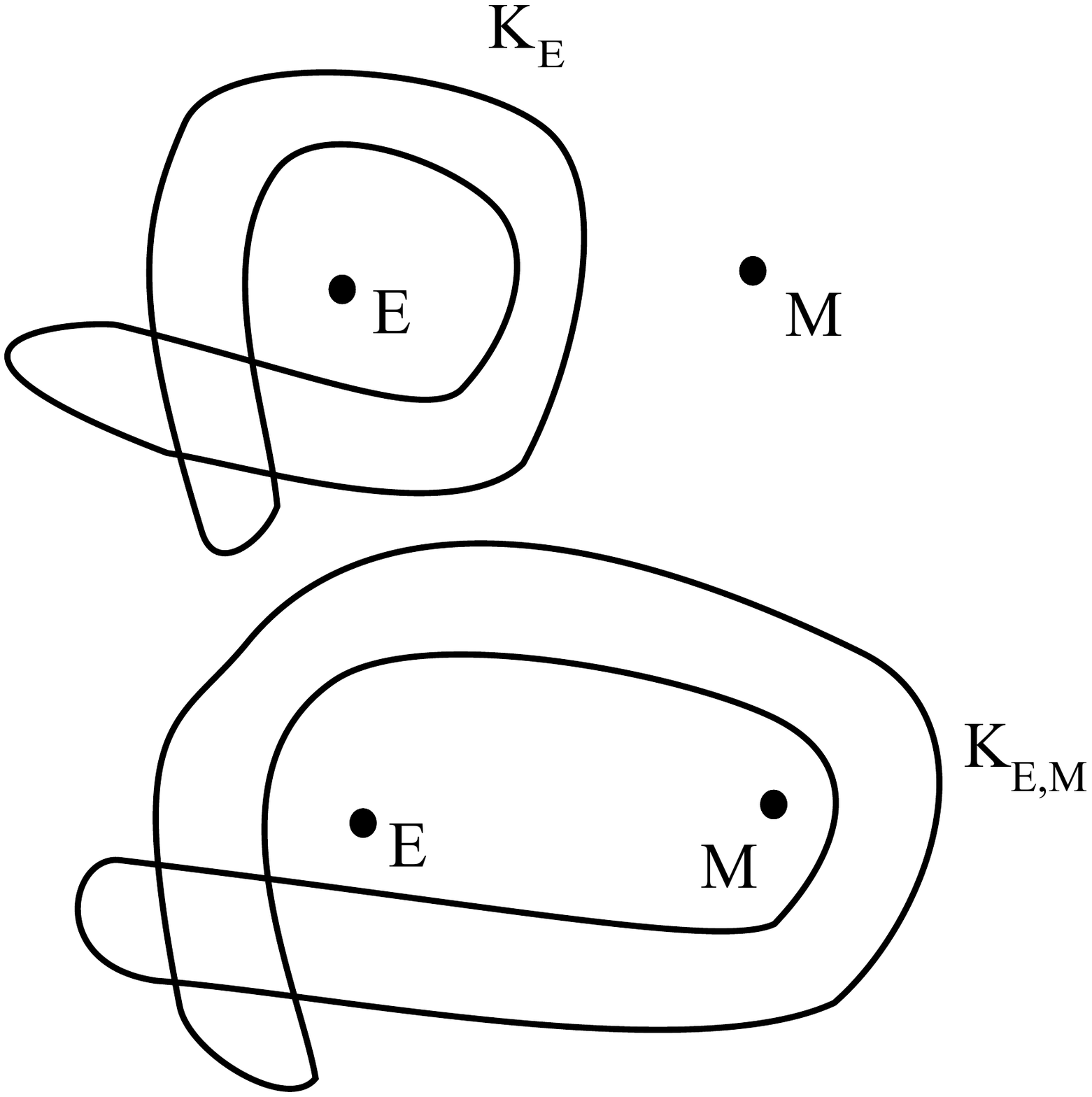}
\caption{Two loops that are not distinguishable by one-center invariants}
\label{Fig: Ke Kem}
\end{figure}




\begin{ex}\label{ex:Kn}
Generalizing Example~\ref{ex:non-conn-sum}, consider for $n\in\N$ the immersed loop $K^n\subset\C\setminus\{E,M\}$ winding $n$ times counterclockwise around $E$ and $M$ with $n-1$ self-intersections as shown in~\cite[Figure 14]{cieliebak-frauenfelder-koert}. 
Its $J^+$-invariant has been computed in~\cite{cieliebak-frauenfelder-koert} to be $J^+(K^n)=-n(n-1)$. 
Suppose now that $n=2m$ is even. Then one component $\tilde K^n$ of the preimage of $K^n$ under the Levi-Civita map at $0$ (or equivalently at $E$ or $M$) is diffeomorphic to $K^m$, so it has $w_{M_1}(\tilde K^n)=w_{M_2}(\tilde K^n)=m$ and $J^+(\tilde K^n)=J^+(K^m)=-m(m-1)$. Hence we can read off the invariants
\begin{align*}
   \JJ_0(K^n) &= J^+(K^n)+n^2/2+n^2/2 = -n(n-1)+n^2 = n, \cr
   \JJ_E(K^n) &= J^+(\tilde K^n) + m^2/2 + m^2/2 = -m(m-1)+m^2 = m, \cr
   \JJ_M(K^n) &= m, \cr
   \JJ_{E,M}(K^n) &= J^+(K^n) = -n(n-1).
\end{align*}
Note the the four invariants sum up to
$$
   (J_0+J_E+J_M+J_{E,M})(K^n) = n+n-n(n-1) = n(3-n).
$$
\end{ex}

The following lemma describes the remainders mod $2$ of the four $J^+$-like invariants. 

\begin{lemma}\label{lem:parities}
The invariant $\JJ_{E,M}(K)$ is always an even integer mod $2n(K)$. The remainders mod $2$ of the other three invariants $\JJ_0,\JJ_E,\JJ_M$ depend on the free homotopy class $[K]$ modulo the moves $(I_E)$ and $(I_M)$ and are given in Table~\ref{tab:parities}. The invariant $n(K)$ has value $0$ for $[K]=e$ and $[K]=m$, and value $n$ for $[K]=(em)^n$.
\end{lemma}

\begin{table}[h!]
  \begin{center}
    \caption{Values of the invariants mod $2$}
\medskip
    \label{tab:parities}
    \begin{tabular}{l|c|c|c} 
      Class $[K]$ & $J_0$ & $J_E$ & $J_M$ \\
      \hline
      $e$ & $1/2$ & $0$ & $1/2$ \\
      $m$ & $1/2$ & $1/2$ & $0$ \\
      $(em)^n$, $n\equiv 0$ mod $4$ & $0$ & $0$ & $0$ \\
      $(em)^n$, $n\equiv 2$ mod $4$ & $0$ & $1$ & $1$ \\
      $(em)^n$, $n\equiv 1$ mod $2$ & $1$ & $1$ & $1$ \\
    \end{tabular}
  \end{center}
\end{table}

Note that the invariants $\JJ_0,\JJ_E,\JJ_M$ detect the free homotopy classes $e$ and $m$, and for the classes $(em)^n$ they detect the parity of $n$ mod $2$ and satisfy the relation
\begin{equation}\label{eq:parities}
   \JJ_E\equiv\JJ_M\equiv n/2 \text{ mod }2 \quad \text{if } n \text{ is even.}
\end{equation}

\begin{proof}
The invariant $\JJ_{E,M}$ takes values in $2\Z/2n\Z$ because $J^+$ takes values in $2\Z$. 
For the other three invariants $\JJ_0,\JJ_E,\JJ_M$, note first that they all change by multiples of $2$ under a $(II_+)$ move and under addition of small loops, so their parities ($=$ remainders mod $2$) remain unchanged under arbitrary free homotopies as well as the moves $(I_E)$ and $(I_M)$. Therefore, is suffices to compute the parities for some representatives of the classes in Lemma~\ref{lem:top-loops}(b). We represent the classes $e$, $m$ and $1$ by small circles around $E$, $M$ and $0$, respectively, and the class $(em)^n$ for $n\in\N$ by the loop $K^n$ in~\cite[Figure 14]{cieliebak-frauenfelder-koert} winding $n$ times around both $E$ and $M$. On these loops one easily reads off the parities of the invariants $\JJ_0,\JJ_E,\JJ_M$ from their definitions. 
\end{proof}

\subsection{Relations among the four invariants}

In the preceding subsections we have defined four invariants: $\mathcal{J}_{0}$ for the non-regularized system, $\mathcal{J}_{E}$ and $\mathcal{J}_{M}$ for the partially regularized systems, and the pair $(\mathcal{J}_{E, M}, n)$ for the Birkhoff-regularized system. In this subsection we will analyze relations between these invariants. 
Crucial ingredients are Propositions 6 and 7 from~\cite{cieliebak-frauenfelder-koert} as well as the following construction. 

{\bf Interior connected sum. }
Let $K_1,K_2\subset\C\setminus\{0\}$ be disjoint generic immersed oriented loops meeting the following
requirements:
\begin{description}
 \item[(i)] $0$ and $K_1$ lie in the unbounded component of $\mathbb{C}\setminus K_2$;
 \item[(ii)] $K_2$ lies in the component $C$ of $\C\setminus K_1$ containing $0$.
\end{description}
See Figure \ref{fig:int-conn-sum}. 
\begin{figure}
\center
\includegraphics[width=100mm]{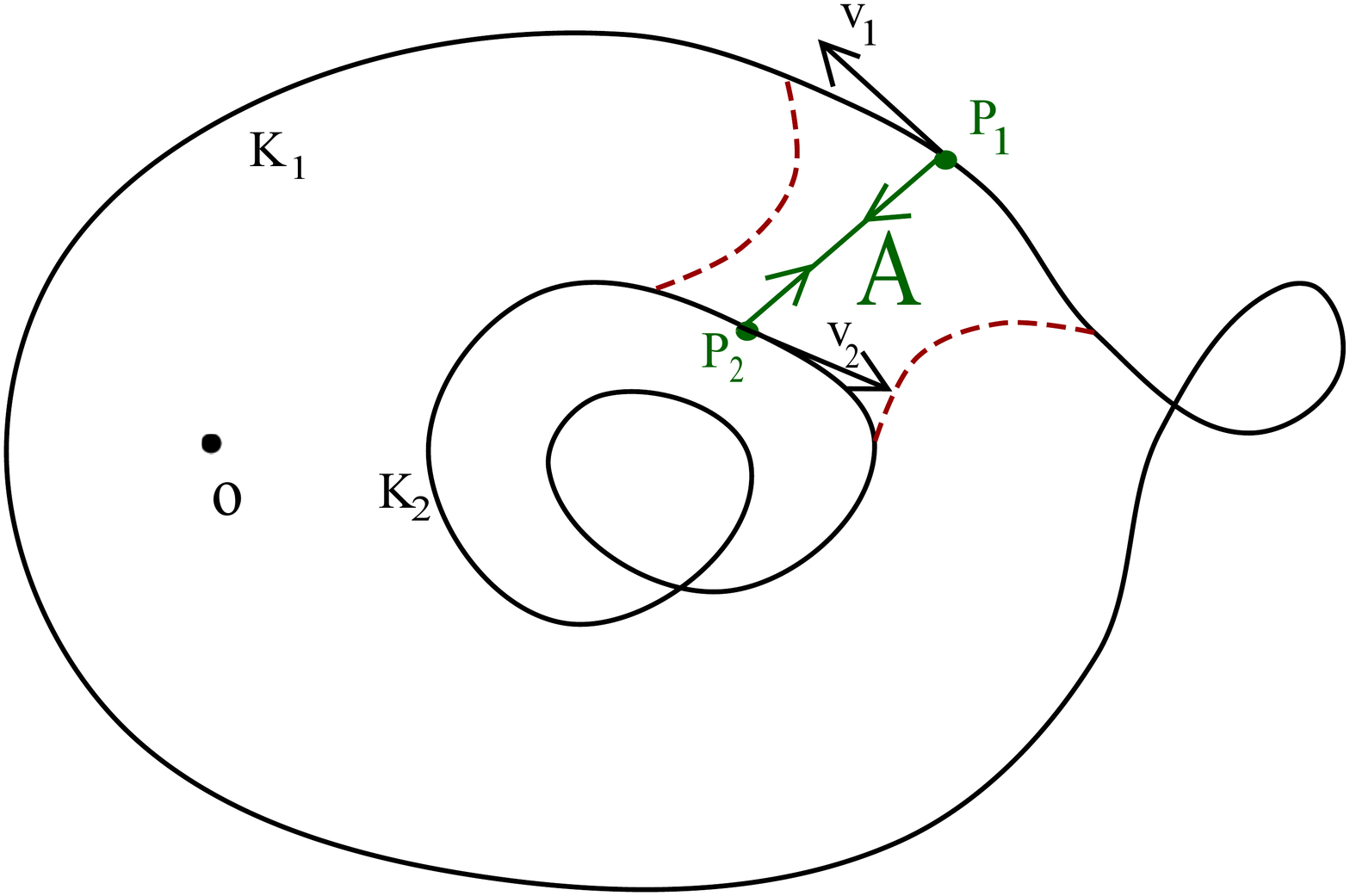}
\caption{Interior connected sum}
\label{fig:int-conn-sum}
\end{figure}
Suppose there exists an embedded arc $A$ connecting two non-double points $p_1\in K_1$ and $p_2\in K_2$ such that $A\setminus\{p_1,p_2\}\subset C\setminus K_2$ and the pairs $(v_1,n_1)$ and $(v_2,n_2)$ are positive bases, where $v_i$ is the velocity vector of $K_i$ at $p_i$ and $n_i$ a vector pointing into the interior of $A$ at its endpoint $p_i$. Then the {\em interior connected sum} $K_1\#_iK_2$ is defined by connecting $K_1,K_2$ along {two} parallel copies of $A$ and smoothing the corners. The immersion $K_1\#_iK_2$ will in general depend on the choice of the arc $A$. Moreover, for given orientations of $K_1,K_2$ such an arc need not exist. However, such an arc will always exist after pulling an interior arc of $K_1$ and an exterior arc of $K_2$ over themselves through inverse self-tangencies, which does not affect their $J^+$-invariants and winding/rotation numbers. Note that $K_1\#_iK_2$ inherits an orientation from $K_1,K_2$ and its rotation number satisfies
\begin{equation}\label{eq:rot-int-sum}
   r(K_1\#_iK_2) = r(K_1)+r(K_2)+1.
\end{equation}
If the pairs $(v_1,n_1)$ and $(v_2,n_2)$ were negative bases we would get $-1$ instead of $+1$ in this formula. Note that by hypothesis (ii) the inversion $\phi(z)=1/z$ sends $K_2$ to the unbounded component of $\C\setminus\phi(K_1)$. Moreover, from hypothesis (i) we deduce that $\phi(K_1)$ lies in the unbounded component of $\C\setminus\phi(K_2)$. Therefore, $\phi(K_1\#_iK_2)$ is the usual connnected sum 
\begin{equation}\label{eq:int-ext-conn-sum}
   \phi(K_1\#_iK_2)=\phi(K_{1})\# \phi(K_{2}).
\end{equation}
Observe that in the special case where $C$ is the unbounded component of $\C\setminus \ K_1$ the interior connected sum is the usual connected sum.

\begin{cor}\label{cor:int-sum}
For the interior connected sum $K=K_1\#_iK_2$ we have
$$
   J^+(K) = J^+(K_1) + J^+(K_2) - 2w_0(K_1)\,\bigl(r(K_2) + 1\bigr).
$$
In particular, $J^+(K)\equiv J^+(K_1)+J^+(K_2)\mod 2|w_0(K_1)|$. 
\end{cor}

\begin{proof}
Since by hypothesis (i) the point $0$ lies in the unbounded component of $\C \setminus K_2$ it follows that $w_0(K_2)=0$, and therefore $w_0(K)=w_0(K_1)$. By \eqref{eq:int-ext-conn-sum} we have $\phi(K)=\phi(K_1)\#\phi(K_2)$. Replacing $\tilde{K}_1,\tilde{K}_2$ by $K,\phi(K)$ in the identity~\eqref{eq: difference J+} from the proof of Proposition~\ref{prop:JEM} we get
$$   
   J^+\bigl(\phi(K)\bigr) - J^+(K) = 2w_0(K)\,\bigl(r(K) - w_0(K)\bigr).
$$
Using this identity for $K,K_1,K_2$, additivity of $J^+$ under connected sum yields
\begin{align*}
   J^+(K) &= J^+\bigl(\phi(K)\bigr) - 2w_0(K)\,\bigl(r(K) - w_0(K)\bigr) \cr
   &= J^+\bigl(\phi(K_1)\bigr) + J^+\bigl(\phi(K_2)\bigr) - 2w_0(K)\,\bigl(r(K) - w_0(K)\bigr) \cr
   &= J^+(K_1) + 2w_0(K_1)\,\bigl(r(K_1) - w_0(K_1)\bigr) \cr
   &\ \ \ + J^+(K_2) + 2w_0(K_2)\,\bigl(r(K_2) - w_0(K_2)\bigr) - 2w_0(K)\,\bigl(r(K) - w_0(K)\bigr) \cr
   &= J^+(K_1) + J^+(K_2) + 2w_0(K_1)\,\bigl(r(K_1) - w_0(K_1) - r(K) + w_0(K_1)\bigr) \cr
   &= J^+(K_1) + J^+(K_2) - 2w_0(K_1)\,\bigl(r(K_2) + 1\bigr),
\end{align*}
where in the last line we have used~\eqref{eq:rot-int-sum}. 
\end{proof}

{\bf The basic lemma. }
We will also need the following refinement of~\cite[Proposition 7]{cieliebak-frauenfelder-koert}. Let us mention that the proof of~\cite[Proposition 7]{cieliebak-frauenfelder-koert} contained a small gap which we fill in the proof below. 
For a generic immersed loop $K\subset\C^*$ with even winding number $w_0(K)$ we denote by $\wt K$ one component of the preimage of $K$ under the Levi-Civita map $L(z)=z^2$.

\begin{lemma}\label{lem:basic}
On generic immersed loops $K\subset\C^*$, the quadruple of invariants $\bigl(J^+(K),J^+(\wt K),w_0(K),r(K)\bigr)$ attains all values in $2\Z\times 2\Z\times 2\Z\times\Z$. In the case $w_0(K)\neq 0$ we can moreover choose $K$ such that $L^{-1}(K)$ can be deformed to two disjoint curves contained in the left/right half-planes by a regular homotopy in $\C$ undergoing only inverse self-tangencies. 
\end{lemma}

\begin{proof}
Let $w\in 2\Z$ be a given even winding number. 
Let $K^w\subset\C^*$ be any generic immersion with $w_0(K^w)=w$ possessing two adjacent parallel arcs $A_1,A_2$ oriented in the same direction such that the path in $K^w$ from $A_1$ to $A_2$ winds an odd number of times around the origin. It has invariants 
$$\bigl(J^+(K^w),J^+(\wt{K^w})\bigr)=(2a,2b)$$
for some $a,b\in\Z$. A $(II^+)$ move pulling $A_1$ across $A_2$ increases $J^+(K^w)$ by $2$ and leaves $J^+(\wt{K^w})$ unchanged because the two new double points in $K^w$ do not give rise to double points in $\wt{K^w}$. Performing $k\in\N_0$ such operations, we obtain an immersion $K^w_{k}$ with invariants 
$$J^+(K^w_{k}) = 2a+2k\quad\text{and}\quad J^+(\wt{K^w_{k}}) = 2b.$$ 
Next we take the connected sum $K^w_{k,\ell}$ of $K^w_{k}$ and an immersion $K'$ with $w_0(K')=0$ and $J^+(K')=2\ell$, for any $\ell\in\Z$. Its lift $\wt{K^w_{k,\ell}}$ under the Levi-Civita covering is the connected sum of $\wt{K^w_{k}}$ and $K'$, so by additivity of $J^+$ we get the invariants 
\begin{equation}\label{eq:invariants}
J^+(K^w_{k,\ell}) = 2a+2k+2\ell\quad\text{and}\quad J^+(\wt{K^w_{k,\ell}}) = 2b+2\ell.
\end{equation} 
By appropriate choices of $k\in\N_0$ and $\ell\in\Z$ we can arrange arbitrary values in $2\Z\times 2\Z$ for the pair $\bigl(J^+(K^w_{k,\ell}),J^+(\wt{K^w_{k,\ell}})\bigr)$. Moreover, we can prescribe the rotation number of $K'$ to arrange the desired rotation number for $K^w_{k,\ell}$. 

Finally, suppose that $w\neq 0$. Then for any $\eps>0$ we can choose $K^w$ to be contained in the strip $[-\eps,\infty)\times[-\eps,\eps]$ such that $K^w\cap[-\eps,1]\times[-\eps,\eps]$ consists of $|w|$ parallel embedded arcs entering and exiting through $\{1\}\times[-\eps,\eps]$ and winding once (positively or negatively depending on the sign of $w$) around the origin. See Figure \ref{fig:strip}. 
\begin{figure}
\center
\includegraphics[width=100mm]{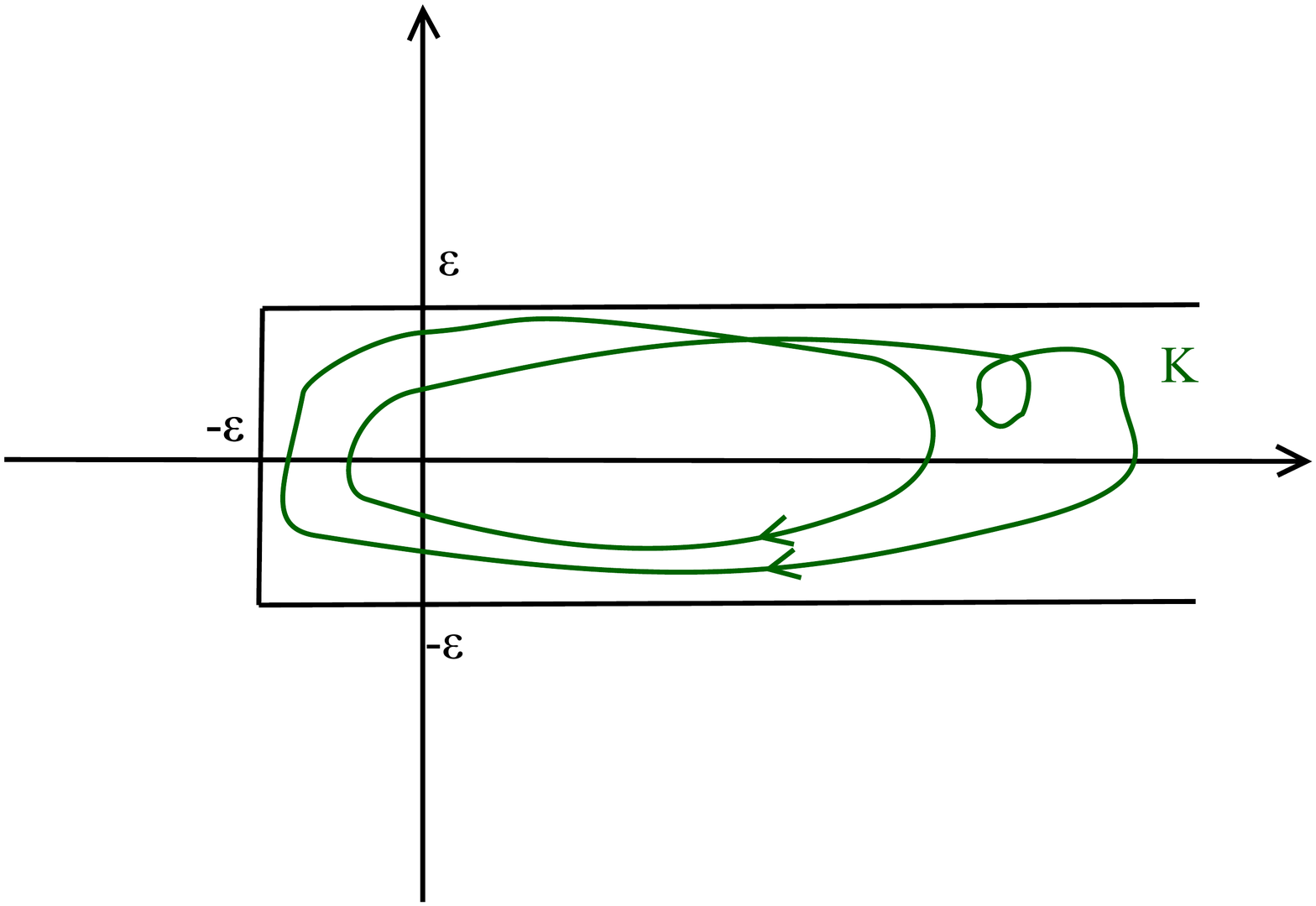}
\caption{Loop contained in a strip}
\label{fig:strip}
\end{figure}
(Note that for $w=0$ this is not possible because of the condition on the parallel arcs $A_1,A_2$.) The modifications above can be performed outside the rectangle $[-\eps,1]\times[-\eps,\eps]$ so that the resulting loop $K=K^w_{k,\ell}$ still has the same property. It follows that $L^{-1}(K)=\wt K\cup(-\wt K)$, where $\wt K\subset[-\sqrt{\eps},\infty)\times[-\sqrt{\eps},\sqrt{\eps}]$ is diffeomorphic to $K$, so $\wt K$ and $-\wt K$ can be disjoined by a regular homotopy in $\C$ undergoing only inverse self-tangencies. 
\end{proof}

Now we are ready to discuss the relations among the invariants. 
Since the parities of the winding numbers $w_E,w_M$ around $E,M$ do not change under Stark-Zeeman homotopies, we distinguish four cases. Recall that $\JJ_E$ is always even and the parities of $\JJ_0$, $\JJ_E$, $\JJ_M$ and $n$ are determined by those of $w_E,w_M$ via Lemmas~\ref{lem:n} and~\ref{lem:parities}.  

{\bf The case $w_E,w_M$ even. }
%
%
By Lemmas~\ref{lem:n} and~\ref{lem:parities}, in this case $n$ is even and $(\JJ_0,\JJ_E,\JJ_M,\JJ_{E,M})\in 2\Z\times\Z\times\Z\times2\Z/2n\Z$ satisfy relation~\eqref{eq:parities}. 

\begin{prop}\label{prop:even-even-mod4}
On generic immersed loops in $\C\setminus\{E,M\}$ with $w_E,w_M$ even the four invariants satisfy the relation
\begin{equation}\label{eq:mod4}
   \JJ_0+\JJ_E+\JJ_M+\JJ_{E,M}\equiv n\mod 4
\end{equation}
(which makes sense modulo $2n$ because in this case $2n$ is divisible by $4$).
\end{prop}

\begin{proof}
For a generic immersed loop $K\subset\C\setminus\{E,M\}$ with $w_E,w_M$ even we denote by $L_E^{-1}(K)^1$, $L_M^{-1}(K)^1$, $B^{-1}(K)^1$ one connected component of the preimage of $K$ under the Levi-Civita maps and $E$, $M$ and the Birkhoff map, respectively.  

To prove relation~\eqref{eq:mod4}, we first claim that the invariant $L:=\JJ_0+\JJ_E+\JJ_M+\JJ_{E,M}$ does not change modulo $4$ under a $(II^+)$ move on $K$. To see this, let $A$ be an arc in $K$ connecting the two points involved in the direct self-tangency. We distinguish $4$ cases according to the parities of the winding numbers $w_E(A),w_M(A)$ of $A$ around $E,M$. \\
If $w_E(A)$ and $w_M(A)$ are even
the direct self-tangency induces direct self-tangencies on $L_E^{-1}(K)^1$, $L_M^{-1}(K)^1$ and $B^{-1}(K)^1$, so $L$ increases by $8$. \\  
If $w_E(A)$ is even and $w_M(A)$ odd 
the direct self-tangency induces direct a self-tangency on $L_E^{-1}(K)^1$ but not on $L_M^{-1}(K)^1$ and $B^{-1}(K)^1$, so $L$ increases by $4$. \\  
If $w_E(A)$ is odd and $w_M(A)$ even
the direct self-tangency induces direct a self-tangency on $L_M^{-1}(K)^1$ but not on $L_E^{-1}(K)^1$ and $B^{-1}(K)^1$, so $L$ increases by $4$. \\  
If $w_E(A)$ and $w_M(A)$ are odd
the direct self-tangency induces direct a self-tangency on $B^{-1}(K)^1$ but not on $L_E^{-1}(K)^1$ and $L_M^{-1}(K)^1$, so $L$ increases by $4$. \\  
This proves the claim, which implies that the equivalence class of $L$ mod $4$ does not change under arbitrary regular homotopies of $K$ in $\C\setminus\{E,M\}$. It also does not change under the moves $(I_E)$ and $(I_M)$ through collisions at $E$ resp.~$M$ which homotopically replace a loop around $E$ resp.~$M$ by its inverse. By Lemma~\ref{lem:top-loops}(b) the free homotopy classes of loops in $\C\setminus\{E,M\}$ with even winding numbers around $E$ and $M$ modulo the moves $(I_E)$ and $(I_M)$ are in bijection to conjugacy classes $[(em)^n]$ with $n\in\N_0$ even, where $e,m$ correspond to loops around $E,M$ respectively. We can represent the conjugacy class $[(em)^n]$ by the immersed loop $K^n$ in Example~\ref{ex:Kn}. By Lemma~\ref{lem:top-loops}(c) we can therefore connect $K$ by a regular homotopy in $\C\setminus\{E,M\}$ together with moves $(I_E)$ and $(I_M)$ to the loop $K^n$, for some even $n\in\N_0$, with some loops attached to the outermost strand of $K^n$ to arrange the correct rotation number. It was computed in Example~\ref{ex:Kn} that $L(K^n)=n(3-n)\equiv n$ mod $4$, so relation~\eqref{eq:mod4} holds for $K^n$. Attaching a loop to the outermost strand of $K^n$ from the outside/inside results in attaching a similar loop to the lifts of $K^n$ under $L_E$, $L_M$ and $B$. An attachment from the outside is a $(I_\infty)$ move which leaves the four invariants (and thus $L$) unchanged. By Proposition~\ref{prop:J+}(b), an attachment from the inside decreases each of the four invariants by $2$ and thus does not change $L$ mod $4$. Hence $L(K)\equiv L(K^n)\equiv n$ mod $4$ and relation~\eqref{eq:mod4} is proved.
\end{proof}

{\em Remark. }The end of the preceding proof could be shortened by connecting $K$ by a regular homotopy to any generic immersed loop $K_0$ located outside a large disk containing $E,M$ and appealing to the proof of Proposition~\ref{prop:even-even} below to conclude $L(K)\equiv L(K_0)\equiv n$ mod $4$.
\medskip

The following proposition shows that, except for relation~\eqref{eq:mod4}, the invariants $\JJ_0,\JJ_E,\JJ_M,\JJ_{E,M}$ are completely independent. 

\begin{prop}\label{prop:even-even}
There exist generic immersed loops in $\C\setminus\{E,M\}$ with arbitrarily prescribed values of the invariants 
$$
   (\JJ_0,\JJ_E,\JJ_M,\JJ_{E,M},n,w_E,w_M,r) \in 2\Z\times \Z\times \Z\times 2\Z/2n\Z\times 2\N_0\times 2\Z\times 2\Z\times \Z
$$
satisfying relations~\eqref{eq:parities} and~\eqref{eq:mod4}. 
\end{prop}

\begin{proof}
Using Lemma~\ref{lem:basic}, we pick an immersion $K_E\subset D_E\setminus\{E\}$ located in a small disk $D_E$ around $E$ with prescribed invariants
$$
   \bigl(\JJ_0(K_E),\JJ_E(K_E),w_E(K_E),r(K_E)\bigr) = (j_E^1,j_E^2,w_E,r_E)\in 2\Z\times 2\Z\times 2\Z\times\Z.
$$
(Note that $\JJ_0(K_E)=J^+(K_E)+w_E(K_E)^2/2$ and $\JJ_E(K_E)=J^+(\tilde K_E)$ for a component $\tilde K_E$ of its lift under the Levi-Civita map around $E$.) 
Similarly, we pick an immersion $K_M\subset D_M\setminus\{M\}$ located in a small disk $D_M$ around $M$ with prescribed invariants
$$
   \bigl(\JJ_0(K_M),\JJ_M(K_M),w_M(K_M),r(K_M)\bigr) = (j_M^1,j_M^2,w_M,r_M)\in 2\Z\times 2\Z\times 2\Z\times\Z.
$$
Finally, we pick an immersion $K_0\subset\C\setminus D_0$ located outside a large disk $D_0$ around the origin containing $D_E\cup D_M$ with prescribed invariants
$$
   \bigl(\JJ_0(K_0),\JJ^+(\tilde K_0),w_0(K_0),r(K_0)\bigr) = (j_0^1,j_0^2,w_0,r_0)\in 2\Z\times 2\Z\times 2\Z\times\Z,
$$
where $\tilde K_0$ denotes one component of the preimage of $K_0$ under the map $z\mapsto z^2$. 
Note that
$$
   w_E(K_M)=w_M(K_E)=0. 
$$
Consider now the iterated interior connected sum
$$
   K := (K_0\#_iK_E)\#_iK_M. 
$$
(Recall that the interior connected sum can be defined after possibly modifying $K_0,K_E,K_M$ without changing their invariants, and it depends on choices, which will be irrelevant for the following discussion.) 
This is a generic immersed loop in $\C\setminus\{E,M\}$ whose invariants we now compute. In view of~\eqref{eq:rot-int-sum}, its winding and rotation numbers are
$$
   w_E(K)=w_0+w_E,\quad w_M(K)=w_0+w_M,\quad r(K) = r_0+\rho,
$$
where we abbreviate
$$
   \rho:= r_E+r_M+2. 
$$
Next note that
\begin{gather*}
   J^+(K_0) = j_0^1 - w_E(K_0)^2/2 - w_M(K_0)^2/2 = j_0^1 - w_0^2, \cr
   J^+(K_E) = j_E^1 - w_E^2/2,\qquad J^+(K_M) = j_M^1 - w_M^2/2.
\end{gather*}
Using this and Corollary~\ref{cor:int-sum} we compute
\begin{align*}
   J^+(K) &= J^+(K_0)+J^+(K_E)+J^+(K_M)-2w_0\rho, \cr
   \JJ_0(K) &= J^+(K) + w_E(K)^2/2 + w_M(K)^2/2 \cr
   &= J^+(K_0) +J^+(K_E)+J^+(K_M)-2w_0\rho + (w_0+w_E)^2/2 + (w_0+w_M)^2/2 \cr
   &= j_0^1 + j_E^1 + j_M^1 + w_0(w_E+w_M-2\rho). 
\end{align*}
Let us denote by $L_E^{-1}(K)^1$ one component of the preimage of $K$ under the partial regularization map at $E$, and similarly for $K_0,K_E,K_M$. Since all winding numbers around $E$ are even, we can choose the preimages such that 
$$
   L_E^{-1}(K)^1 = L_E^{-1}(K_0)^1 \#_i L_E^{-1}(K_E)^1 \#_i L_E^{-1}(K_M)^1.
$$
Let us write 
$$
   w_0 = 2\bar w_0.
$$
Then $L_E^{-1}(K_0)^1$ winds around both preimages $M_1,M_2$ with winding number $\bar w_0$ while $L_E^{-1}(K_M)^1$ only winds with winding number $w_M$ around one of them, say $M_1$, so 
$$
   w_{M_1}\bigl(L_E^{-1}(K)^1\bigr) = \bar w_0+w_M,\qquad w_{M_2}\bigl(L_E^{-1}(K)^1\bigr) = \bar w_0.
$$
Since $L_E^{-1}(K_0)^1$ is isotopic to the component $\tilde K_0$ of the preimage of $K_0$ under the map $z\mapsto z^2$, using Corollary~\ref{cor:int-sum} we find
\begin{align*}
   J^+\bigl(L_E^{-1}(K)^1\bigr) &= J^+\bigl(L_E^{-1}(K_0)^1\bigr) + J^+\bigl(L_E^{-1}(K_E)^1\bigr) + J^+\bigl(L_E^{-1}(K_M)^1\bigr) - 2\bar w_0\rho, \cr
   &= j_0^2 + j_E^2 + J^+(K_M) - 2\bar w_0\rho, \cr
   \JJ_E(K) &= J^+\bigl(L_E^{-1}(K)^1\bigr) + w_{M_1}\bigl(L_E^{-1}(K)^1\bigr)^2/2 + w_{M_2}\bigl(L_E^{-1}(K)^1\bigr)^2/2 \cr
   &= j_0^2 + j_E^2 + J^+(K_M) - 2\bar w_0\rho + (\bar w_0+w_M)^2/2 + \bar w_0^2/2 \cr
   &= j_0^2 + j_E^2 + j_M^1 + \bar w_0(\bar w_0+w_M-2\rho).
\end{align*}
Switching the roles of $E,M$ gives
$$
   \JJ_M(K) = j_0^2 + j_E^1 + j_M^2 + \bar w_0(\bar w_0+w_E-2\rho).
$$
Finally, let $B^{-1}(K)^1$ be one component of the preimage of $K$ under the Birkhoff regularization map, and similarly for $K_0,K_E,K_M$. Again we can choose the preimages such that 
$$
   B^{-1}(K)^1 = (B^{-1}(K_0)^1 \#_i B^{-1}(K_E)^1) \#_i B^{-1}(K_M)^1.
$$
Since the preimages of $K_E,K_M$ do not wind around the origin, we have
$$
   w_0\bigl(B^{-1}(K)^1\bigr) = w_0,\qquad n(K)=|w_0|. 
$$
Since $B$ looks like $L_E$ near $E$, the curve $B^{-1}(K_E)^1$ is located near $E$ and isotopic to $L_E^{-1}(K_E)^1$, thus $J^+\bigl(B^{-1}(K_E)^1\bigr) = j_E^2$ and similarly $J^+\bigl(B^{-1}(K_M)^1\bigr) = j_M^2$. On the other hand, near infinity $B$ is a disconnected 2-to-1 covering, so 
$J^+\bigl(B^{-1}(K_0)^1\bigr) = J^+(K_0) = j_0^1-w_0^2$. Using this and Corollary~\ref{cor:int-sum} we find
\begin{align*}
   J^+\bigl(B^{-1}(K)^1\bigr) &= J^+\bigl(B^{-1}(K_0)^1\bigr) + J^+\bigl(B^{-1}(K_E)^1\bigr) + J^+\bigl(B^{-1}(K_M)^1\bigr) - 2w_0\rho \cr
   &= j_0^1 - w_0^2 + j_E^2 + j_M^2 - 2w_0\rho \cr 
   &= j_0^1 + j_E^2 + j_M^2 - w_0(w_0+2\rho). 
\end{align*}
Let us now choose the rotation numbers $r_E,r_M$ such that $\rho=0$. With this simplification, the winding and rotation numbers of $K$ are
$$
   \bigl(n(K),w_E(K),w_M(K),r(K)\bigr) = \bigl(|w_0|,w_0+w_E,w_0+w_M,r_0\bigr).
$$
We see that by choosing $w_0,w_E,w_M,r_0$ we can arrange arbitrary values in $2\N_0\times 2\Z\times 2\Z\times\Z$ for this quadrupel of numbers. Fixing these choices, the four $J^+$-type invariants (still with $\rho=0$) were computed to be
\begin{align*}
   \JJ_0(K) &= j_0^1 + j_E^1 + j_M^1 + w_0(w_E+w_M), \cr 
   \JJ_E(K) &= j_0^2 + j_E^2 + j_M^1 + \bar w_0(\bar w_0+w_M), \cr
   \JJ_M(K) &= j_0^2 + j_E^1 + j_M^2 + \bar w_0(\bar w_0+w_E), \cr
   \JJ_{E,M}(K) &\equiv j_0^1 + j_E^2 + j_M^2 - w_0^2 \mod 2n(K). 
\end{align*}
Not taking the last equation modulo $2n(K)$, we view this as a system of $4$ inhomogeneous linear equations in $6$ variables $j_0^i,j_E^i,j_M^i$ ($i=1,2$) which we can choose freely in $2\Z$. Taking the second and third equations mod $2$ yields $\JJ_E(K)\equiv\JJ_M(K)\equiv n(K)^2/4\equiv n(K)/2$ mod $2$, so relation~\eqref{eq:parities} holds. Adding up the $4$ equations yields
$$
   \JJ_0(K)+\JJ_E(K)+\JJ_M(K)+\JJ_{E,M}(K)\equiv \bar w_0(2\bar w_0+w_M+w_E) \equiv n(K)^2/2\equiv n(K)$$
modulo $4$, so relation~\eqref{eq:mod4} holds as well. Inspection of the integer $4\times 6$ matrix defining the equations shows that by choosing the $6$ variables $j_0^i,j_E^i,j_M^i$ ($i=1,2$) we can change $(\JJ_0(K),\JJ_E(K),\JJ_M(K),\JJ_{E,M}(K))$ by any quadruple of even integers $(a_0.a_E,a_M,a_{E,M})$ satisfying $a_0+a_E+a_M+a_{E,M}\equiv 0$ mod $4$, and therefore arrange any values compoatible with relations~\eqref{eq:parities} and~\eqref{eq:mod4}. 
\end{proof}

{\bf The case $w_E$ odd, $w_M$ even. }
We now discuss the case with $w_{E}$ odd, $w_M$ even. The results carry over to the case $w_E$ odd, $w_M$ even by switching the roles of $E$ and $M$. By Lemmas~\ref{lem:n} and~\ref{lem:parities}, in this case $n=0$ and the invariants take values $(\JJ_0,\JJ_E,\JJ_M,\JJ_{E,M})\in(2\Z+1/2)\times 2\Z\times(2\Z+1/2)\times 2\Z$. 

We begin with the following refinement of~\cite[Proposition 6]{cieliebak-frauenfelder-koert}: 

\begin{prop}\label{prop: w_{E} odd: J_{E} J_{0}} 
For a generic immersed loop $K\subset\C\setminus\{E,M\}$ with $w_{E}(K)$ odd we have 
$$\mathcal{J}_{E} (K) =2 \mathcal{J}_{0}(K) -1.$$ 
If in addition $w_M(K)$ is even, then $\JJ_E(K)$ and $\JJ_{E,M}(K)$ are both divisible by $4$. 
\end{prop}

\begin{proof} 
Temporarily forgetting the singularity $M$ and applying~\cite[Proposition 6]{cieliebak-frauenfelder-koert} to the curve $K$ with $w_E(K)$ odd we get
$$J^{+} (\tilde{K}_{E}) = 2 \Bigl(J^{+} (K) + \dfrac{w_{E}^{2}(K)}{2}\Bigr) -1,$$
where $\tilde K_E$ is one component of the preimage of $K_E$ under the Levi-Civita map at $E$. 
Thus
$$J^{+} (\tilde{K}_{E}) +w_{M}^{2} (E) = 2 (J^{+} (K) + \dfrac{w_{E}^{2}(K)}{2}+\dfrac{w_{M}^{2}(K)}{2}) -1.$$
The left hand side is $\mathcal{J}_{E} (K)$ by Lemma~\ref{lem:JE}, and the right hand side is $2 \mathcal{J}_{0}(K) -1$ by the definition of $\mathcal{J}_{0} (K)$. This proves the first assertion. 

Suppose now that in addition $w_M(K)$ is even. Then divisibility of $\JJ_E(K_E)$ by $4$ follows from $\JJ_E(K)=2\JJ_0(K)-1$ and $\JJ_0(K_E)\in 2\Z+1/2$. 
For the last assertion, first note that a $(II^+)$ move on $K$ corresponds to two $(II^+)$ moves on $B^{-1}(K)$ and therefore increases $\JJ_{E,M}(K)$ by $4$. Hence the equivalence class of $\JJ_{E,M}(K)$ mod $4$ does not change under arbitrary regular homotopies of $K$ in $\C\setminus\{E,M\}$. It also does not change under the moves $(I_E)$ and $(I_M)$ through collisions at $E$ resp.~$M$ which change the winding numbers around $E$ resp.~$M$ by $\pm 2$. Now the free homotopy classes of loops in $\C\setminus\{E,M\}$ modulo the moves $(I_E)$ and $(I_M)$ are in bijection to $\Z_2\times\Z_2$, classified by their winding numbers $w_E$ and $w_M$ mod $2$. Since $w_E(K)$ is odd and $w_M(K)$ is even, and a homotopy between immersed loops in the plane with the same rotation number can be $C^0$-approximated by a regular homotopy, we can connect $K$ by a regular homotopy in $\C\setminus\{E,M\}$ together with moves $(I_E)$ and $(I_M)$ to a generic immersion $K_E$ located near $E$ with $w_E(K_E)=1$ and $w_M(K_E)=0$. By the preceding discussion we have $\JJ_{E,M}(K_E)\equiv\JJ_{E,M}(K)$ mod $4$, and $\JJ_{E,M}(K_E) = \JJ_E(K_E)$ is divisible by $4$ by the first assertion.
\end{proof}

So $\JJ_E$ is determined by $\JJ_0$ and it remains to study the invariants $(\mathcal{J}_{0}, \mathcal{J}_{M}, \mathcal{J}_{E,M})\in (2\Z+1/2)\times(2\Z+1/2)\times 4\Z$. We begin with the following (much simpler) analogue of Lemma~\ref{lem:basic} for odd winding number. 

\begin{lemma}\label{lem:4.14 odd}
For any given $(j_E,w_E,r_E)\in (2\Z+1/2)\times (2\Z+1)\times\Z$ there exists a generic immersed loop $K_E\subset\C\setminus\{E,M\}$ located in a small disk around $E$ with 
$$
   \bigl(\JJ_0(K_E),w_E(K_E),r(K_E)\bigr) = (j_E,w_E,r_E).
$$
\end{lemma}

\begin{proof}
Begin with a loop with the desired winding number $w_E$, and take the connected sum with another loop with $w_E=0$ and prescribed $J_0$ to arrange the desired $J_0$. Finally, take a further connected sum with a loop with prescribed rotation number and $J^+=w_E=0$ to arrange the desired rotation number. 
\end{proof}

We will also need the following easy lemma on rotation numbers. 

\begin{lemma}\label{lem:rot-number}
Let $K\subset\C^*$ be an immersed loop with winding number $w_0(K)$ around the origin. If $w_0(K)$ is odd the rotation numbers of $K$ and its lift under the Levi-Civita map $L(z)=z^2$ are related by
$$
   r\bigl(L^{-1}(K)\bigr) = 2r(K) - w_0(K).
$$
If $w_0(K)$ is even the rotation numbers of $K$ and one component $L^{-1}(K)^1$ of its lift under the Levi-Civita map are related by
$$
   r\bigl(L^{-1}(K)^1\bigr) = r(K) - w_0(K)/2.
$$
\end{lemma}

\begin{proof}
After a regular homotopy we may assume that $K$ consists a $w_0(K)$-fold covered circle around $0$ with $r':=r(K)-w_0(K)$ contractible circles in $\C^*$ attached. If $w_0(K)$ is odd, then $L^{-1}(K)$ consists of a $w_0(K)$-fold covered circle around $0$ with $2r'$ contractible circles in $\C^*$ attached, so its rotation number is $r\bigl(L^{-1}(K)\bigr) = w_0(K) + 2r' = 2r(K) - w_0(K)$. 
If $w_0(K)$ is even, then $L^{-1}(K)^1$ consists of a $w_0(K)/2$-fold covered circle around $0$ with $r'$ contractible circles in $\C^*$ attached, so its rotation number is $r\bigl(L^{-1}(K)^1\bigr) = w_0(K)/2 + r' = r(K) - w_0(K)/2$. 
\end{proof}

The following proposition shows that for $w_E$ odd and $w_M$ even, the invariants $\JJ_0,\JJ_M,\JJ_{E,M}$ satisfy no further relations. 

\begin{prop}\label{prop:even-odd}
There exist generic immersed loops in $\C\setminus\{E,M\}$ with arbitrarily prescribed values of the invariants
$$
   (\JJ_0,\JJ_M,\JJ_{E,M},w_E,w_M,r) \in (2\Z+1/2)\times (2\Z+1/2)\times 4\Z\times (2\Z+1)\times 2\Z\times \Z.
$$
\end{prop}


\begin{proof}
As in the proof of Proposition~\ref{prop:even-even} we construct $K$ as the iterated interior connected sum
$$K := (K_0\#_iK_E)\#_iK_M$$
of a loop $K_E$ near $E$, $K_M$ near $M$, and $K_0$ outside a large disk containing $E$ and $M$. By Lemma~\ref{lem:4.14 odd} we can prescribe the invariants
$$\bigl(\JJ_{0} (K_{E}), w_{E}(K_{E}), r(K_{E})\bigr)=(j_{E}, w_{E}, r_{E}) \in (2\Z+1/2) \times (2 \Z +1) \times \Z$$ 
and by Lemma~\ref{lem:basic} we can prescribe the invariants
$$\bigl(\JJ_{0} (K_{M}),\JJ_{M} (K_{M}), w_{M}(K_{E}), r(K_{M})\bigr)=(j^{1}_{M}, j^{2}_{M}, w_{M}, r_{M}) \in 2\Z \times 2\Z \times 2\Z  \times \Z,$$
$$\bigl(\JJ_{0} (K_{0}),\JJ^{+} (\tilde{K}_{0}), w_{0}(K_{0}), r(K_{0})\bigr)=(j^{1}_{0}, j^{2}_{0}, w_{0}, r_{0}) \in 2\Z \times 2\Z \times 2\Z  \times \Z,$$
where $\tilde K_0$ denotes one component of the preimage of $K_0$ under the Levi-Civita map $L(z)=z^2$. As in the proof of Proposition~\ref{prop:even-even} we obtain
$$
   w_E(K)=w_0+w_E,\quad w_M(K)=w_0+w_M,\quad r(K) = r_0+\rho,\quad \rho:= r_E+r_M+2
$$
and (since $w_0$ and $w_M$ are even)
\begin{gather*}
   J^{+} (K_{0})=j_{0}^{1}-w_{0}^{2},\quad 
   J^{+}(K_{E})=j_{E}-w_{E}^{2}/2,\quad 
   J^{+}(K_{M})=j_{M}^{1}-w_{M}^{2}/2, \cr
   \JJ_0(K) = j_0^1 + j_E + j_M^1 + w_0(w_E+w_M-2\rho),\cr
   \JJ_M(K) = j_0^2 + j_E + j_M^2 + w_0(w_0+w_E-2\rho).
\end{gather*}
To compute $\JJ_{E.M}(K)$, let $B^{-1}(K_{0})^{1,2}$ and $B^{-1}(K_{M})^{1,2}$ be the connected components of the preimages of $K_0$ resp.~$K_M$ under the Birkhoff map $B:\C^*\to\C$. Here we label $B^{-1}(K_{0})^{1}$ the component inside the unit disk and by $B^{-1}(K_{0})^{2}$ the one outside. 
We choose $w_M\neq 0$ and arrange for $K_{M}$ the additional property in Lemma~\ref{lem:basic} that  the two components of $B^{-1}(K_{M})$ can be disjoined by a regular homotopy involving only inverse self-tangencies. We label $B^{-1}(K_{M})^{1}$ the component that is connected to $B^{-1}(K_{0})^{1}$ by the connected sum construction, and by $B^{-1}(K_{M})^{2}$ the one connected to $B^{-1}(K_{0})^{2}$. Then the preimage $B^{-1}(K)$ looks like in Figure \ref{fig:odd-even}.
\begin{figure}
\center
\includegraphics[width=100mm]{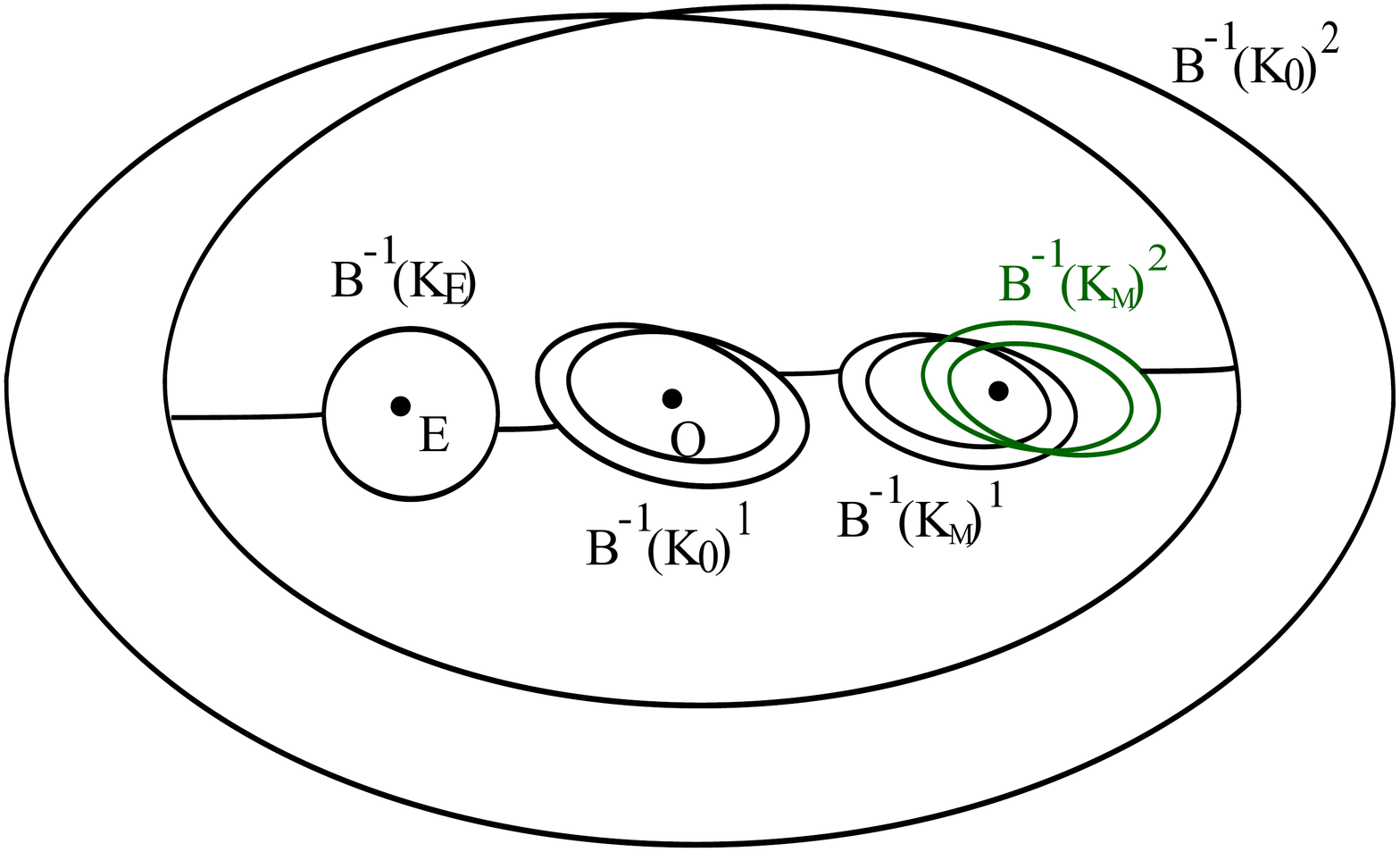}
\caption{The case $w_E$ odd, $w_M$ even}
\label{fig:odd-even}
\end{figure}
Disjoining the two components of $B^{-1}(K_{M})$ in $B^{-1}(K)$ through inverse self-tangencies and pushing $B^{-1}(K_0)^1$ away from $0$ does not change $J^+$, so it leads to a curve $K'$ with $\JJ^+(K')=\JJ_{E,M}(K)$ which can be written as an iterated connected/interior connected sum
$$K'=K'_2\#_iK'_1$$ 
with
$$K_2'=B^{-1}(K_{0})^{2} \#_i  B^{-1}(K_{M})^{2},\quad 
K_1'=\bigl(B^{-1}(K_{E})\#B^{-1}(K_{0})^{1}\bigr)\# B^{-1}(K_{M})^{1}.$$
Note that since the interior connected sums are formed by positive bases, the mirrored connected sums involved are also formed by positive bases, so their rotation numbers obey formula~\eqref{eq:rot-int-sum}.

To compute $J^{+}(K')$ (and thus $\JJ_{E, M} (K)$), recall that the Birkhoff map behaves like the map $z\mapsto z/2$ near infinity and like the respective Levi-Civita maps near $E$ and $M$. In particular, $B^{-1}(K_{0})^{2}$ is diffeomorphic to $K_{0}$ and thus has the same invariants. Using this and Lemma~\ref{lem:rot-number}, we compute the rotation numbers
$$
   r\bigl(B^{-1}(K_0)^{2}\bigr)=r_0,\quad 
   r\bigl(B^{-1}(K_E)\bigr)=2r_E-w_E,\quad 
   r\bigl(B^{-1}(K_M)^{1,2}\bigr)=r_M-w_M/2.
$$
From Lemma~\ref{lem:rot-win} we infer
$$r(B^{-1}(K_0)^1)=r_0-2w_0,$$
whence in view of formula~\eqref{eq:rot-int-sum}
\begin{align*}
   r(K_1')
   &= r(B^{-1} (K_{E}))+r(B^{-1}(K_{0})^{1}) + r(B^{-1}(K_{M})^{1})+2  \cr 
   &= 2r_{E}-w_{E}+r_{0}-2w_0+r_M-w_M/2+2.
\end{align*}
Using repeatedly equation~\eqref{eq: difference J+}, Corollary~\ref{cor:int-sum} and Proposition~\ref{prop: w_{E} odd: J_{E} J_{0}}, we now compute the $J^+$-invariants:
\begin{align*} 
   J^+\bigl(B^{-1}(K_0)^2\bigr) &= J^{+}(K_{0}) = j_0^1-w_0^2, \cr 
   J^+\bigl(B^{-1}(K_M)^{1,2}\bigr) &= j_M^2, \cr 
   J^+(K_2') &= J^+\bigl(B^{-1}(K_0)^2\bigr) + J^+\bigl(B^{-1}(K_M)^2\bigr) \cr
   &\ \ \ - 2w_0\bigl(B^{-1}(K_0)^2\bigr) \bigl(r(B^{-1}(K_M)^2)+1\bigr) \cr
   &= j_0^1-w_0^2+j_M^2-2w_0(r_M-w_M/2+1), \cr
   J^+\bigl(B^{-1}(K_E)\bigr) &= 2J^+(K_E)+w_E(K_E)^2-1 = 2j_E-w_E^2+w_E^2-1 \cr
   &= 2j_E-1, \cr
   J^+\bigl(B^{-1}(K_0)^1\bigr) &= J^+\bigl(B^{-1}(K_0)^2\bigr)+2w_0(r_0-w_0) \cr
   &= j_0^1-w_0^2+2w_0(r_0-w_0) = j_0^1-3w_0^2+2w_0r_0, \cr
   J^+(K_1') &= J^+\bigl(B^{-1}(K_E)\bigr) + J^+\bigl(B^{-1}(K_0)^1\bigr) + J^+\bigl(B^{-1}(K_M)^1\bigr) \cr
   &= 2j_E-1 + j_0^1-3w_0^2+2w_0r_0 + j_M^2, \cr
   \JJ_{E,M}(K) &= J^+(K') = J^+(K_2') + J^+(K_1') - 2w_0(K_2')\bigl(r(K_1')+1\bigr) \cr
   &= j_0^1-w_0^2+j_M^2-2w_0(r_M+1)+w_0w_M \cr
   &\ \ \ + 2j_E-1+j_0^1-3w_0^2+2w_0r_0+j_M^2 \cr
   &\ \ \ -2w_0(2r_{E}-w_{E}+r_{0}+r_M-2w_0-w_M/2+2+1) \cr
   &= 2j_0^1+2j_E+2j_M^2+2w_0(w_E+w_M)-4w_0\rho-1.
\end{align*}
Let us now choose the rotation numbers $r_E,r_M$ such that $\rho=0$. With this simplification, the winding and rotation numbers of $K$ are
$$
   \bigl(w_E(K),w_M(K),r(K)\bigr) = \bigl(w_0+w_E,w_0+w_M,r_0\bigr).
$$
We see that by fixing some $w_M\neq 0$ (which was needed above in order to apply Lemma~\ref{lem:basic}) and varying $w_0,w_E,r_0$ we can arrange arbitrary values in $\Z\times 2\Z\times\Z$ for this triple of numbers. Fixing these choices, the three $J^+$-like invariants (still with $\rho=0$) were computed to be
\begin{align*}
   \JJ_0(K) &= j_0^1 + j_E + j_M^1 + w_0(w_E+w_M) \in 2\Z+1/2, \cr
   \JJ_M(K) &= j_0^2 + j_E + j_M^2 + w_0(w_0+w_E)\in 2\Z+1/2, \cr
   \JJ_{E,M}(K) &= 2j_0^1+2j_E+2j_M^2+2w_0(w_E+w_M)-1\in 4\Z.
\end{align*}
We view this as a system of $3$ inhomogeneous linear equations in $5$ variables $(j_E,j_M^1,j_M^2,j_0^1,j_0^2)\in(2\Z+1/2)\times 2\Z\times 2\Z\times 2\Z\times 2\Z$ which we can choose freely. 
Inspection of the integer $3\times 5$ matrix defining the equations shows that by varying $(j_E,j_M^1,j_M^2,j_0^1,j_0^2)$ we can change $(\JJ_0(K),\JJ_M(K),\JJ_{E,M}(K))$ by any triple in $2\Z\times 2\Z\times 4\Z$, and therefore arrange any values in $(2\Z+1/2)\times (2\Z+1.2)\times 4\Z$. 
\end{proof}


{\bf The case $w_E,w_M$ odd. }
By Lemmas~\ref{lem:n} and~\ref{lem:parities}, in this case $n$ is odd and $(\JJ_0,\JJ_E,\JJ_M,\JJ_{E,M})\in (2\Z+1)\times (2\Z+1)\times (2\Z+1)\times2\Z/2n\Z$. Moreover, Proposition~\ref{prop: w_{E} odd: J_{E} J_{0}} immediately implies

\begin{cor} 
If $w_{E} (K)$ and $w_{M} (K)$ are both odd, then
$$\mathcal{J}_{E}(K)=\mathcal{J}_{M}(K)=2 \mathcal{J}_{0}(K) -1. $$
\end{cor}

So $\JJ_E,\JJ_M$ are determined by $\JJ_0$. The following proposition shows that $\mathcal{J}_{0}$ and $\mathcal{J}_{E, M}$ satisfy no further relations. 

\begin{prop}\label{prop:odd-odd}
There exist generic immersed loops in $\C\setminus\{E,M\}$ with arbitrarily prescribed values of the invariants
$$
   (\JJ_0,\JJ_{E,M},n,w_E,w_M,r) \in (2\Z+1)\times 2\Z/2n\Z\times (2N_0+1)\times (2\Z+1)\times (2\Z+1) \times \Z.
$$
\end{prop}

\begin{proof}
As in the proof of Proposition~\ref{prop:even-even} we construct $K$ as the iterated interior connected sum
$$K := (K_0\#_iK_E)\#_iK_M$$
of a loop $K_E$ near $E$, $K_M$ near $M$, and $K_0$ outside a large disk containing $E$ and $M$. We choose the winding number $w_0(K_0)$ odd and the winding numbers $w_E(K_E),w_M(K_M)$ even. Then by Lemma~\ref{lem:4.14 odd} we can prescribe the invariants
$$\bigl(\JJ_0(K_0), w_0(K_0), r(K_0)\bigr)=(j_0^1, w_0, r_0) \in (2\Z+1) \times (2\Z +1) \times \Z$$ 
and by Lemma~\ref{lem:basic} we can prescribe the invariants
$$\bigl(\JJ_0(K_E),\JJ_E(K_E), w_E(K_E), r(K_E)\bigr)=(j^1_E, j^2_E, w_E, r_E) \in 2\Z \times 2\Z \times 2\Z  \times \Z,$$
$$\bigl(\JJ_0(K_M),\JJ_M(K_M), w_M(K_M), r(K_M)\bigr)=(j^1_M, j^2_M, w_M, r_M) \in 2\Z \times 2\Z \times 2\Z  \times \Z.$$
As in the proof of Proposition~\ref{prop:even-even} we obtain
$$
   w_E(K)=w_0+w_E,\quad w_M(K)=w_0+w_M,\quad r(K) = r_0+\rho,\quad \rho:= r_E+r_M+2
$$
and (since $w_E$, $w_M$ are even and the parity of $w_0$ played no role in the computation of these two invariants)
\begin{align*}
   \JJ_0(K) &= j_0^1 + j_E^1 + j_M^1 + w_0(w_E+w_M), \cr 
   \JJ_{E,M}(K) &\equiv j_0^1 + j_E^2 + j_M^2 - w_0^2 \mod 2n(K),
\end{align*}
where $n(K)=|w_0|$ and we have again chosen $r_E,r_M$ such that $\rho=0$. Hence by varying $(w_0,w_E,w_M,r_0)$ we can arrange arbitrary values for 
$$\bigl(n(K),w_E(K),w_M(K),r(K)\bigr) \in (2N_0+1)\times (2\Z+1)\times (2\Z+1) \times \Z,$$
and given these, by varying $(j_0^1,j_E^1,j_E^2,j_M^1,j_M^2)$ we can arrange arbitrary values for $(\JJ_0,\JJ_{E,M})\in (2\Z+1)\times 2\Z/2n\Z$. 
\end{proof}

\section{Further discussions}

\subsection{Knot types and Legendrian Knots}
As in the one-center case discussed in~\cite{cieliebak-frauenfelder-koert}, each periodic orbit of a two-center Stark-Zeeman system describes an oriented knot in the Moser-regularized energy hypersurface $\Sigma_c^M\cong\R P^3\#\R P^3$, and each generic immersion $K\subset\C\setminus\{E,M\}$ lifts (by adding its tangent direction) to an oriented knot in $\gamma\subset\R P^3\#\R P^3$ whose knot type is invariant under Stark-Zeeman homotopies. Note that according to Lemma~\ref{lem:parities} the free homotopy class of $\gamma$ is captured by the invariants $\JJ_E(K)$, $\JJ_M(K)$, and $n(K)$. The proof of~\cite[Corollary 3]{cieliebak-frauenfelder-koert} shows that every oriented knot type in $\R P^3\#\R P^3$ is realized by a Moser regularized periodic orbit in some two-center Stark-Zeeman system. A periodic orbit in $\Sigma_c^M\cong\R P^3\#\R P^3$ can be further lifted to an oriented knot in the Birkhoff regularized energy hypersurface $\Sigma_c^B\cong S^1\times S^2$ whose knot type is also invariant under Stark-Zeeman homotopies of its footpoint projection. 

As mentioned in~\cite{cieliebak-frauenfelder-koert}, it would be interesting to search for more refined invariants under one- or two-center Stark-Zeeman homotopies using invariants of their Legendrian lifts (by adding the unit conormal vectors). 

\subsection{$N$-center Stark-Zeeman systems}
The notions of planar $1$- and $2$-center Stark-Zeeman systems generalize in the obvious way to that of a planar $N$-center Stark-Zeeman system. On a given energy level, a partial Levi-Civita regularization at some subset of the $N$ centers can be defined by by going to a Riemann surface branched at these centers, see Klein and Knauf~\cite{klein-knauf}. This should give rise to $2^N$ different $J^+$-like invariants for periodic orbits of a planar $N$-center Stark-Zeeman system, which would be interesting to be further explored.



\end{document}